\documentclass[11pt,twoside,a4paper]{article}
\usepackage{latexsym}
\usepackage{amsthm,amssymb,amsbsy,amsmath,amsfonts,amssymb,amscd}
\usepackage{graphicx,dsfont}
\usepackage{caption,color,xcolor}
\usepackage{bbm}
\usepackage{hyperref}
\usepackage{ulem}
\newcommand{\commentout}[1]{}
\newcommand {\fer}   {\eqref}

\newcommand {\e}  {\varepsilon}

\newcommand {\Chi} {{\bf \raise 2pt \hbox{$\chi$}} }
\newcommand {\f}   {\frac}
\newcommand {\p}   {\partial}

\newcommand{\beq}{\begin{equation}}
\newcommand{\eeq}{\end{equation}}
\newcommand{\bal}{\begin{align}}
\newcommand{\bc}{\begin{cases}}
\newcommand{\ec}{\end{cases}}
\newcommand{\bea} {\begin{array}{rl}}
\newcommand{\eea} {\end{array}}
\newcommand{\bepa}{\left\{ \begin{array}{l}}
\newcommand{\eepa} {\end{array}\right.}
\newtheorem{theorem}{Theorem}[section]
\newtheorem{lemma}[theorem]{Lemma}
\newtheorem{cor}[theorem]{Corollary}

\newtheorem{prop}[theorem]{Proposition}
\newtheorem{rem}[theorem]{Remark}

\def\be{\begin{eqnarray}}
\def\ee{\end{eqnarray}}
\def\ben{\begin{eqnarray*}}
\def\een{\end{eqnarray*}}

\numberwithin{equation}{section}
\numberwithin{figure}{section}

\def\be{\begin{eqnarray}}
\def\ee{\end{eqnarray}}

\renewcommand{\P}{\mathbb{P}}

\def\me{\medskip\noindent}

\newcommand{\Co}{\mathcal{C}}

\newcommand{\adh}{\mathrm{cl}}

\def\D{\mathbb{D}}

\def\N{\mathbb{N}}
\def\P{\mathbb{P}}
\def\R{\mathbb{R}}
\def\E{\mathbb{E}}

\def\ind{{\mathchoice {\rm 1\mskip-4mu l} {\rm 1\mskip-4mul}
{\rm 1\mskip-4.5mu l} {\rm 1\mskip-5mu l}}}

\def\11{\mathbbm{1}}
\def\dSko{d_{\text{Sko}}}

\title{From stochastic individual-based models to free-boundary Hamilton-Jacobi equations}

\author{
 Nicolas Champagnat\thanks{Universit\'e de Lorraine, CNRS, Inria, IECL, F-54000 Nancy, France; E-mail: \texttt{nicolas.champagnat@inria.fr}}
\and {Sylvie M\'el\'eard}\thanks{Ecole Polytechnique, CNRS, Institut polytechnique de Paris, Inria, route de Saclay, 91128 Palaiseau Cedex-France; E-mail: \texttt{sylvie.meleard@polytechnique.edu}}  \and Sepideh Mirrahimi\thanks{Univ Toulouse, INSA Toulouse, CNRS, IMT, Toulouse, France; E-mail: \texttt{sepideh.mirrahimi@math.univ-toulouse.fr}} \and Viet Chi Tran \thanks{Univ. Lille, CNRS, Inria, UMR 8524 - Laboratoire Paul Painlev\'e, F-59000 Lille, France; E-mail: \texttt{chi.tran@inria.fr}}
 }
\date{\today}

\begin{document}
\maketitle
\pagestyle{plain}
\pagenumbering{arabic}

 \begin{abstract}
 We study a stochastic branching model for a population structured by a quantitative phenotypic trait and subject to births, deaths, and mutations. In a regime of large population and small mutations, and in logarithmic scales of size and time, we derive a certain class of free boundary Hamilton-Jacobi equations with state constraints from the stochastic individual-based system. This goes beyond the classical Hamilton-Jacobi equations obtained from deterministic models by taking into account the possible extinction of the system in certain regions of the trait space. The proof is obtained by combining methods for the analysis of Hamilton-Jacobi equations with probabilistic tools from the theory of large deviations and branching processes.
 \end{abstract}

\me Keywords: stochastic birth-death model, measure-valued process, large population approximation, 
mutation, Hamilton-Jacobi equations, large deviations.

\bigskip

\section{Introduction}
\label{sec:intro}

In the mathematical modeling of eco-evolutionary dynamics of phenotypically structured populations with mutation and selection, several methods have been developed over the past two decades, based either on stochastic or deterministic approaches.
Among these approaches, some of them aim to describe long-term evolutionary dynamics of dominant evolutionary paths in large populations, based either on stochastic processes or deterministic equations.

The stochastic approach developed in~\cite{champagnat06,champagnatmeleard2011} was based on birth and death measure-valued processes~\cite{fourniermeleard,champagnatferrieremeleard} involving mutation and competition. The mutation  time scale was assumed  much slower than that of demographic events and, using slow-fast time scales arguments,  some dominant evolutionary dynamics were highlighted  in the time scale of mutations. However, the hypothesis of very rare mutations, which may be reasonable for certain phenotypes, may seem unrealistic for others, as may be the very long timescale required to describe evolutionary phenomena~\cite{WaxmanGavrilets2005}. 

At the same time, an approach involving Hamilton-Jacobi equations has been developed to characterize the  evolutionary dynamics~\cite{diekmannjabinmischlerperthame,Barles2009}, as limits of integro-differential selection-mutation equations under the assumption of small mutation effect and long time, and using a Hopf-Cole transformation. These integro-differential equations themselves have been derived as large population limits of stochastic individual-based models~\cite{fourniermeleard,champagnatferrieremeleard}. The order in which these limiting procedures are applied can influence the outcome, since they do not always commute, potentially leading to some inaccuracies in the deterministic approximations. In particular, the large population approximation prevents phenomena of local extinction, which may lead to an overestimation of the speed of evolution, as this may be slowed down by certain extinction events \cite{perthame2010}.

More recently~\cite{durrettmayberry,Bovier2019,champagnatmeleardtran2021,coquillekrautsmadi,esserkraut2024}, stochastic birth and death mutation-selection models were studied under different parameter scalings, with rare mutations, but not so rare as in the previous stochastic approach. They were inspired by some stochastic Hopf-Cole transformation, making them much closer to the deterministic models described above.\\

In this work, following  our seminal work \cite{champagnat23} (see also \cite{Jeddi}),  we reconcile the stochastic and deterministic  approaches by deriving a new class of  state-constrained Hamilton-Jacobi equations directly from the stochastic system in the time and size scales of \cite{champagnatmeleardtran2021} but assuming small mutations instead of rare mutations. In~\cite{champagnat23}, we studied the case of a uniformly supercritical branching population with discretized and compact state space using a direct convergence approach and deterministic methods on Hamilton-Jacobi equations~\cite{Barles2009}. In~\cite{Jeddi}, this result was extended to weaker assumptions in the supercritical case and new results were obtained in the subcritical case, allowing extinction of the population.    
Here we develop a new and robust mathematical strategy to study the general case where the growth rate can change sign, without requiring a discretized trait space, and where the restrictive assumptions required in \cite{champagnat23,Jeddi} are relaxed. This leads to a new limiting object which takes into account the effect of demographic stochasticity and possible extinction of small subpopulations. This approach yields more accurate and informative results with regard to modeling. In particular, the speed of evolution can be slowed down compared to classical Hamilton Jacobi approaches, making this model relevant for understanding the evolutionary dynamics of advantageous traits.

Our main result is based on the variational representation of the solution of the Hamilton-Jacobi equations and on exponential deviations results for branching processes. We were inspired by an analytical approach proposed in \cite{Mirrahimi2012} and by the study of branching processes in~\cite{berestyckibrunetharrisharrisroberts,mallein}. \\

In Section~\ref{sec:intro}, we develop the model and state our main results (Theorems \ref{thm:bornesup} and \ref{thm:borneinf}), that is large deviations estimates for historical processes on a logarithmic time scale and on the path space, characterized by a  variational expression.
The latter is related to the solution of state constrained Hamilton-Jacobi equations (Theorem \ref{prop:var-to-visc}  and Corollary \ref{cor:lineage}). In Section \ref{sec:biblio}, we compare our results with existing work in different contexts, all seeking to characterize optimal trajectories in one way or another. Based on Feynman-Kac formula for branching processes, we study in Section~\ref{sec:LDP} the underlying Markovian mutational process for which large deviation results can be easily obtained. We complete these results by proving some uniformity with respect to the initial condition in the large deviation principle. Sections~\ref{sec:in-expectation},~\ref{sec:upper-bound} and~\ref{sec:lower-bound} are devoted to the proofs of Theorems \ref{thm:bornesup} and \ref{thm:borneinf}. The proof of the lower bound (Theorem~\ref{thm:borneinf}) is more difficult and we use moment methods developed for branching processes as in \cite{mallein} to control the number of particles whose paths are in a given tube, following the strategy used in \cite{berestyckibrunetharrisharrisroberts}. Section 6 studies the link between the variational formulation of the limit and the Hamilton-Jacobi equation.

\bigskip
Notation : The space of finite measures on $\R$ is denoted by ${\cal M}_F(\R)$ and the set of finite point measures on $\R$ by ${\cal M}_P(\R)$. For a measure $\mu$ and a $\mu$-integrable or positive function $f$, we write $\int_{\R}f(x) \mu(dx)=\langle \mu,f\rangle$.

 For a Polish space $E$ (in particular for $E=\R$ or $E=\mathcal{M}_F(\R)$), we will denote by $\mathbb{D}([0,t],E)$ the Skorohod space of right-continuous and left-limited (càdlàg) functions from $[0,t]$ to $E$.
 We will sometimes use the notation $\D[0,t]$ or $\D$ instead of $\D([0,t],\R)$ and $\D(\R_+,\R)$ respectively. 
 The space $\D[0,t]$ is equipped with the Skorohod distance  $\dSko$ that makes this space Polish, see e.g. \cite{billingsley}: for $f,g\in \D[0,t]$,
 \begin{equation}
     \label{def:dSko}
     \dSko(f,g)= \inf_{\lambda} \Big\{\sup_{s\in [0,t]} |f\circ \lambda(t)-g (t)|+ \sup_{s\not=r} \big|\log \frac{\lambda(s)-\lambda(r)}{s-r}\big|\Big\},
 \end{equation}where the supremum is taken over all increasing homeomorphisms of $[0,t]$.
 We denote by $C[0,t]$ the set of continuous functions from $[0,t]$ to $\R$, by $C^\infty[0,t]$ the set of infinitely differentiable real functions on $[0,t]$ and by $AC[0,t]$ the set of absolutely continuous real functions on $[0,t]$.

\subsection{Model}
\label{sec:model}

We consider a stochastic birth-death-mutation process describing an asexual population of individuals (for example, cells or bacteria), characterized by a quantitative trait $x$ belonging to $\R$. Notice that for better legibility we restrict our work to  $x\in \R$. However, our results can easily be extended to $x\in \R^d$.
We introduce a parameter  $K\in \N$ scaling the initial population size and the mutation amplitude. 
We assume that an individual with trait $x\in \R$ undergoes the following events independently from the other individuals: this individual can
\begin{itemize}
    \item give birth to a new individual with the same trait $x$ at rate $b(x)$;
    \item die at rate $d(x)$;
    \item give birth to a mutant individual at rate $p(x)$ and the mutant trait is given by $x+\frac{Y}{\log K}$ where $Y$ is distributed as $G(y)dy$ where $G$ is a probability density function. 
\end{itemize}

We assume that $G:\R\to\R_+^*$ is positive continuous, that it is an even function and that it has all its exponential moments finite. 

\medskip
We also assume that   the functions $b, d, p$ are nonnegative
locally Lipschitz-continuous bounded functions, and that there exist positive constants $\bar b$, $\bar p$ and $\underline{p}$ such that for all $x\in \R$,
 \beq
 \label{borne-bp}
0\le b(x)\leq \bar b, \qquad \underline{p}\leq p(x) \leq \bar p.
 \eeq   
 In particular, there exists a positive constant $\bar{R}$ such that the growth rate satisfies for all $x\in \R$
 \beq
 \label{def:R}
 R(x):=b(x)+p(x)-d(x) \le \bar{R}.
 \eeq
 
In the model, we index individuals using the Ulam-Harris-Neveu numbering. The set of labels is 
\[ 
\mathcal{U}= \bigcup_{n\geq 0} \left(\N^*\times\{0,1\}^n\right),
\]
where we use the usual notation $v_1v_2\ldots v_n$ for the vector $(v_1,v_2,\ldots,v_n)\in{\cal U}$. Individuals initially present in the population form the first generation and are labeled by integers from $1$ to $N_0^K$, the initial number of individuals. When an individual of label $v\in \mathcal{U}$ reproduces, we consider that it acquires the label $v0$ and its offspring is given the label $v1$.

\medskip

The population dynamics at scale $K$ is the point measure-valued process $(Z^K_t)_{t\geq 0} $ defined for each $t\ge 0$ by
$$Z^K_t = \sum_{v\in V^K_t} \delta_{X^{K,v}_{t}},$$
where $V^K_t$ denotes the labels of individuals alive at time $t$ and each individual $v$ alive at time $t$ has the trait $X^{K,v}_t$. In the sequel, we will sometimes also denote by $V^K_{[0,t]}$ the set of labels for individuals born before time $t$ (including those still alive at time $t$), i.e.\ $V^K_{[0,t]}=\bigcup_{s\in[0,t]}V^K_s$.\\

We  assume that 
\begin{equation}\label{conditioninitiale}
Z^K_0 \mbox{ is a Poisson point measure on $\R$ with intensity measure } K^{\beta^K_0(x)}dx,
\end{equation}
where for any $K$, $\beta^K_0$ is a continuous function on $\R$ which converges uniformly (as $K$ tends to infinity) to a function $\beta_0$ such that there exist constants $\bar{\beta}$ and $\alpha>0$ such that
\begin{equation}
    \label{as:beta0}
    \beta_0\text{ is locally Lipschitz on $\R$ and }\beta_0(x)\le \bar{\beta}-\alpha |x|,\ \forall x\in\R.
\end{equation} 
This implies in particular that the intensity measure $K^{\beta^K_0(x)}dx$ is finite for any $K$ and hence that the initial number of individuals $N_0^K=\langle Z^K_0,1\rangle$ is almost surely finite, i.e.\ $Z^K_0\in{\cal M}_P(\R)$, since
\begin{equation}
\label{temps0}
E(N_0^K) = \E( \langle Z^K_0,1\rangle) =\int_\R K^{\beta^K_0(x)}dx <+\infty.\end{equation}

\bigskip 
The process $Z^K$ can be represented as the unique strong solution of a stochastic differential equation driven by Poisson point measures. 
Let us consider a Poisson point process $N(ds,dv,d\theta)$ on $\R_+\times \mathcal{U}\times \R_+$ with intensity measure $ds\otimes n(dv)\otimes d\theta$ where $ds$ and $d\theta$ are Lebesgue measures on $\R_+$ and $n(dv)$ is the counting measure on the denumerable set $\mathcal{U}$. We also introduce a Poisson point measure $Q(ds,dv,dy,d\theta)$ on $\R_+\times \mathcal{U}\times \R\times \R_+$  with intensity measure $ds\otimes n(dv)\otimes  G(y)dy\otimes d\theta$. We assume that the random measures $Z^K_0$, $N$ and $Q$ are  independent.  
 
\medskip
Let us consider a test function $\varphi\in \Co_b( \R,\R)$, then
\begin{align}
   \langle Z^K_t,\varphi\rangle= &  \int_\R \varphi(x)Z^K_t(dx)=\sum_{v\in V^K_t} \varphi\big(X^{K,v}_t\big)\nonumber\\
    = & \langle Z^K_0,\varphi\rangle 
     + \int_0^t \int_{\mathcal{U}}\int_{\R_+} \ind_{v\in V^K_{s_-}} \varphi(X^{K,v}_{s_-})\Big( \ind_{\theta\leq b(X^{K,v}_{s_-})} \nonumber\\
     & \hskip 2cm -\ind_{b(X^{K,v}_{s_-})<\theta\leq b(X^{K,v}_{s_-})+d(X^{K,v}_{s_-})}\Big) N(ds,dv,d\theta)\nonumber\\
     + &  \int_0^t\int_{\mathcal{U}} \int_\R \int_{\R_+}
  \ind_{v\in V^K_{s_-},\theta\leq p(X^{K,v}_{s_-})}  \varphi\left(X^{K,v}_{s_-}+\frac{y}{\log K}\right)   Q(ds,dv,dy,d\theta),\label{eq:eds_Poisson}
\end{align}
For $v\in V^K_t$, we can define the lineage of $v$ as follows, and this process will be denoted $(X^{K,v}_{s\wedge t})_{s\geq 0}$. It is a path of $\D(\R_+,\R)$, constant after time $t$ with value $X^{K,v}_t$, and that takes at times $s\leq t$ the trait value of the (unique) ancestor of $v$ (possibly $v$ itself) living at this time. In particular, on the event that individual $v$ reproduces before dying, $X^{K,v}_t=X^{K,v0}_t=X^{K,v1}_t$ for all $t$ (strictly) smaller than the reproduction time of $v$.\\

Using standard It\^o calculus (see \cite{ikedawatanabe,fourniermeleard}), we obtain that:
\begin{align}
\label{semimartingale}
    \langle Z^K_t,\varphi\rangle= &\langle Z^K_0,\varphi\rangle + \int_0^t \langle Z^K_s, (b+p-d) \varphi\rangle \ ds\nonumber\\
    + & \int_0^t\int_{\R} \int_{\R} \big(\varphi(x+\frac{y}{\log K})-\varphi(x)\big) p(x) \,G(y) dy\ Z^K_s(dx)\ ds+M^{K,\varphi}_t,
 \end{align}
 where $M^{K,\varphi}$ is a square integrable martingale  with predictable quadratic variation process:
 \begin{align}
     \langle M^{K,\varphi}\rangle_t= &  \int_0^t \int_{\R} \big(b(x)+p(x)+d(x)\big)\varphi^2(x)\  Z^k_s(dx)\ ds\nonumber\\
     + &
     \int_0^t \int_\R  \int_\R p(x) \, \big(\varphi(x+\frac{y}{\log K})-\varphi(x)\big)^2   G(y)dy\ Z^K_s(dx)\ ds.
 \end{align}
Using \eqref{borne-bp} and~\eqref{conditioninitiale}, it is standard~\cite{fourniermeleard} to prove that for any $T>0$ and $n>0$, 
$$\E(\sup_{t\le T}\langle Z^K_t,1\rangle^n)<+\infty,$$
and then the process $Z^K$ is well defined on any time interval 
$[0,T]$ and 
belongs to $\mathbb{D}(\R_+,{\cal M}_P(\R))$.

\bigskip

We also introduce the associated historical process $(\Theta^K_t, t\ge 0)$ which is a point measure-valued process taking values in ${\cal M}_P(\mathbb{D})$, and  defined for any $t\geq 0$ by
$$ \Theta^K_0 =  Z^K_0\ ;\ \Theta^K_t = \sum_{v\in V^K_t} \delta_{X^{K,v}_{.\wedge t}}.$$
It is possible to write an SDE in the spirit of \eqref{eq:eds_Poisson} to describe the dynamics of the historical process, see \cite{CHMT,HenryMeleardTran,meleardtran}.  

\medskip In the sequel, because the mutation steps are of order of magnitude $1/\log K$, we will consider the process at the time scale  $\log K$. Hence we define the time-changed historical process $\widetilde{\Theta}^K$ taking values in ${\cal M}_P(\mathbb{D})$  for all $t>0$ as
\begin{equation}
\label{histo}
\widetilde{\Theta}^K_t=\sum_{v\in V^K_{t\log K}}\delta_{X^{K,v}_{(.\wedge t)\log K}}.
\end{equation}
For all $T>0$ and all measurable $A\subset\mathbb{D}[0,T]$, we define for all $t\in[0,T]$ the set 
\begin{equation}
    \label{At}
 A_t:=\{f_{\cdot\wedge t},\,f\in A\}
\end{equation}of functions in $A$ stopped at time $t$. 
We define the number of particles living at time $t\log K$ having their lineage in $A$ by
\begin{equation}\label{def:NKA}
N^{K,A}_t=\langle \widetilde{\Theta}^K_{t},\mathbbm{1}_{A_t}\rangle.    
\end{equation}

As in our seminal works \cite{champagnatmeleardtran2021} and \cite{champagnat23}, we are interested in capturing the number of trajectories at logarithmic time scale living in some fixed set, which is of the order of magnitude of a power of $K$. To obtain the limiting dynamics (in $K$) of the exponent of such $K$-power number, we are led to study the asymptotic behaviour of $\log N^{K,A}_t/\log K$, which will be the aim of Theorems \ref{thm:bornesup} and \ref{thm:borneinf}.  Note that these quantities can be seen as the stochastic analog of the Hopf-Cole transformation used by the analysts to describe concentration phenomena (see for example \cite{BP.GB:08} and Section~\ref{sec:biblio} for more details on the related literature), $1/\log K$ playing in our setting the role of $\varepsilon$ in the usual deterministic setting. To achieve this, we use techniques developed in the theory of large deviations and branching Markov processes.

\subsection{Useful mathematical objects and main results}
In this section, we introduce some notations and useful functions that will allow us to state our main results.
We define
\begin{gather}
    H(\alpha)=\int_{\mathbb{R}}(e^{\alpha y}-1)G(y)\,dy, \label{eq:def-H} \\
   L(x,v)=\sup_{\alpha\in\mathbb{R}}\left(\alpha v-p(x)H(\alpha)\right)\label{eq:def-L}. 
\end{gather}
Our assumptions on $G$ imply that $H(\alpha)<\infty$ for all $\alpha\in\R$ and $H$ and $L$ are convex and superlinear functions with respect to $\alpha$ and $v$ (see \cite[Section 3.3]{Evans}). In particular, there exist a positive constant $A$ and a superlinear function $\mu:\R^+\to \R$, with $\lim_{r\to +\infty}\f{\mu(r)}{r}=+\infty$, such that for all $(x,v)\in \R^2$, 
  \beq
  \label{condL}
 \mu(|v|)-A\leq L(x,v),\qquad \lim_{|v'|\to +\infty} \f{\p_v L(x,v')\cdot v'}{|v'|}=+\infty.
  \eeq
Next, for any $f\in\mathbb{D}[0,t]$, we define
\begin{equation}\label{def:Ft-fonctionnelle}
F_t(f)=\beta_0(f(0))+ \int_0^t R(f_s) ds - I_t(f),
\end{equation}
with 
\begin{equation}
\label{def=fonct-action}
I_{t}(f)= 
\begin{cases}
    \int_0^t L(f_s,\dot{f}_s)ds & \text{if }f\in AC[0,t], \\ +\infty & \text{otherwise.}
\end{cases}
\end{equation}
The functions $F_t$ and $I_t$ will play respectively the roles of a cost function and a good rate function associated with a large deviation principle as we will see later in the article. A non-variational formulation of the rate function $I_t$ will be given in Section \ref{sec:I-nonvariational}.\\

\medskip Our main results are asymptotic upper and lower bounds on the logarithm of $N^{K,A}_t$ defined in \eqref{def:NKA}. In the sequel, for any $a\in\R$ and any sequence of real random variables $(X_n)_{n\in\N}$, we will say that 
\[
\limsup_{n\to+\infty} X_n\leq a \quad\text{in probability}
\]
if, for all $\varepsilon>0$,
\[
\lim_{n\to+\infty}\P\left(X_n\geq a+\varepsilon\right)=0.
\]
The extension to $\liminf X_n\geq a$ in probability is straightforward.

\begin{theorem}
\label{thm:bornesup}
    For any $t>0$ and any closed set $A\subset \mathbb{D}[0,t]$, we have, in probability
\[
 \limsup_{K\rightarrow +\infty}   \frac{1}{\log K}\log N^{K,A}_{t} \leq \sup\{F_t(f); f\in A, \; \forall s\in [0,t],\; F_s(f)\geq 0\}. 
\]
\end{theorem}

\begin{theorem}
\label{thm:borneinf}
    For any $t>0$ and any open set $G\subset \mathbb{D}[0,t]$, we have, 
    almost surely,
\[
 \liminf_{K\rightarrow +\infty}   \frac{1}{\log K}\log N^{K,G}_{t} \geq \sup\{F_t(f); f\in G, \; \forall s\in [0,t],\; F_s(f)> 0\}. 
\]
\end{theorem}

Note that, by definition of $I_t(f)$, the upper bound in Theorem~\ref{thm:bornesup} is equal to
\[
\sup\{F_t(f); f\in A\cap AC[0,t], \; \forall s\in [0,t],\; F_s(f)\geq 0\}
\]
and similarly for Theorem~\ref{thm:borneinf}.
Some heuristic explanations for these results are given at the end of this section.

\medskip

To link these results to a Hamilton-Jacobi equation, let us define for $a\in \R$,
\beq
\label{def:u-a}
 u_a(t,x)=\sup\, \{F_t(f); f\in \mathrm{AC}[0,t],\, f(t)=x,\, \forall s\in [0,t], \, F_s(f)\geq a\},
 \eeq
 \[
 \widetilde \Omega_a=\big\{ (t,x)\in [0,+\infty)\times \R; \exists f\in AC[0,t], \, f(t)=x,\,  \forall s\in [0,t], \, F_s(f)\geq a \big\}.
 \]
 Notice that it is immediate that 
 \beq
 \label{upos}
 u_a(t,x)\geq a, \quad \text{for all $(t,x)\in \widetilde\Omega_a$},\qquad \mbox{and }\qquad u_a(t,x)=-\infty,\quad \text{for all $(t,x)\in \widetilde\Omega_a^c$,}
 \eeq
 where we used the convention that $\sup\emptyset = -\infty.$\\
 \medskip

 The next results show that the functions $u_a(t,x)$ allow us to characterize the asymptotic density of individuals at $(t,x)$ in the stochastic population process $Z^K$ and are  solutions of some state-constrained Hamilton-Jacobi equations. 
 
 Let us define for all $t\geq 0$, $x\in\R$ and $\delta>0$
\[
A^{x,\delta}_t=\left\{f\in\mathbb{D}[0,t],\, f(t)\in[x-\delta,x+\delta]\right\},
\]
\[
G^{x,\delta}_t=\left\{f\in\mathbb{D}[0,t],\, f(t)\in(x-\delta,x+\delta)\right\}.
\]
Then, we have the following result.
\begin{theorem}
    \label{cor:HJ}
    For all $t\geq 0$ and $x\in\R$, in probability
     \[
     \lim_{a\downarrow 0}u_a(t,x)\leq
    \liminf_{\delta\to 0}\liminf_{K\to\infty}\frac{1}{\log K}\log N^{K,G^{x,\delta}_t}_t\leq\limsup_{\delta\to 0}\limsup_{K\to\infty}\frac{1}{\log K}\log N^{K,A^{x,\delta}_t}_t\leq u_0(t,x).
    \]
Moreover, if $a\mapsto u_a(t,x)$ is continuous at $a=0$, then, in probability 
    \begin{equation}
    \label{est-density-HJ}
            \lim_{\delta\to 0}\lim_{K\to\infty}\frac{1}{\log K}\log N^{K,A^{x,\delta}_t}_t= \lim_{\delta\to 0}\lim_{K\to\infty}\frac{1}{\log K}\log N^{K,G^{x,\delta}_t}_t=u_0(t,x).
    \end{equation}
\end{theorem}

To define the Hamilton-Jacobi equation, we also need to introduce the set
  \[
  \Omega_a=\{(t,x)\in \widetilde\Omega_a \, |\, u_a(t,x)> a\},
 \]
     
\begin{theorem} 
\label{prop:var-to-visc}
The set $\Omega_a$ is an open set. The function $u_a$ belongs to $C(\adh(\Omega_a))$ and it is the unique locally Lipschitz-continuous and bounded above viscosity solution of the following Hamilton-Jacobi equation  
 \beq
 \label{eq:HJ}
 \begin{cases}
 \p_t u=p(x) H(\p_x u)+R(x),& (t,x)\in \Omega_a\\
 u(t,x)=a,& (t,x)\in \p  \Omega_a , \, t>0,\\
 u(0,x)=\beta_0(x), &\text{for all $x$ s.t $\beta_0(x)>a$}.
 \end{cases}
 \eeq
 
  \end{theorem}The proof of Theorem \ref{prop:var-to-visc} is postponed to Section \ref{section:link-variational-HJ}.\\
  
   Notice  that $u_a$ defined by \eqref{def:u-a}, satisfies a state constraint boundary condition in $\Omega_a$ (\cite[Section 5.1.3]{GB:94}), i.e. $F_s(f)$ has to remain larger than $a$ for all $s\in [0,t]$ in the variational formula~\eqref{def:u-a}. Moreover, $u_a=-\infty$ in $\widetilde{\Omega}^c_a$ and $u_a=a$ in $\widetilde{\Omega}_a\setminus{\Omega_a}$. We lastly show in the following lemma that, for a.e. $a$, $\widetilde{\Omega}_a\setminus{\Omega_a}$ is a Lebesgue-null set and $u_a$ is right continuous with respect to $a$. For this, let us define
  \[
 \Gamma_{a_0}=\bigcup_{a> a_0} \displaystyle\widetilde \Omega_a  .
 \]
One can verify that, for all $a_1,\, a_2\in \R$, with $a_1<a_2$, we have
 \[
 \widetilde{\Omega}_{a_2}\subset  \widetilde{\Omega}_{a_1},\quad   {\Omega_{a_2}}\subset  {\Omega_{a_1}},\quad  u_{a_2}(\cdot,\cdot)\leq u_{a_1}(\cdot,\cdot),
 \]
 and for all $a_0\in \R$,
 \begin{equation}
 \label{cont-ua-Omegac}
   \Gamma_{a_0}\subset  \Omega_{a_0}\subset\widetilde \Omega_{a_0}, \quad\text{and for all $(t,x)\in \widetilde \Omega_{a_0}^c$,} \qquad \lim_{a\downarrow a_0}u_a(t,x)=u_{a_0}(t,x)=-\infty.   
 \end{equation}


 \begin{lemma}
 \label{lem:ua-cont}
   For almost every $a_0\in \R$, we have
      \begin{equation}
   \label{eq:cont-Omega}
          \int_{\R^+\times\R} \mathds{1}_{\widetilde \Omega_{a_0}\setminus\Gamma_{a_0}}(t,x)dtdx=0,
   \end{equation}
   and
   \begin{equation}
       \label{ua-cont}
          \forall (t,x)\in \Gamma_{a_0}\cup \widetilde \Omega_{a_0}^c,\qquad \lim_{a\downarrow a_0} u_a(t,x)=u_{a_0}(t,x).
   \end{equation}
 \end{lemma}
The proof is also deferred to Section \ref{section:link-variational-HJ}.

\begin{rem}Notice that Lemma \ref{lem:ua-cont} is satisfied for almost every $a_0$, but might fail for $a_0=0$. We explain here that up to perturbing slightly the initial condition, we always have that 
for all $(t,x)\in \Gamma_0\cup \Omega_0^c$ and in particular for a.e. $(t,x)$, in probability
      \[
    \lim_{\delta\to 0}\lim_{K\to\infty}\frac{1}{\log K}\log N^{K,A^{x,\delta}_t}_t=u_0(t,x).
    \]The exact formulation will appear in \eqref{eq:perturbation}. To precise this statement, let us define for any $f\in AC[0,t]$
    \[
F_t^\mu(f)=\beta_0(f(0))-\mu+ \int_0^t (b+p-d)(f(s)) ds - I_t(f),
\]
\beq
\label{def:u-mu}
 u_0^\mu(t,x)=\sup\, \{F_t^\mu(f); f\in \mathrm{AC}[0,t],\, f(t)=x,\, \forall s\in [0,t], \, F_s^\mu(f)\geq 0\},
 \eeq
 \[
 \widetilde \Omega_0^\mu=\big\{ (t,x)\in [0,+\infty)\times \R; \exists f\in AC[0,t], \, f(t)=x,\,  \forall s\in [0,t], \, F_s^\mu(f)\geq 0 \big\},
 \]
 \[
  \Omega_0^\mu=\{(t,x)\in \widetilde\Omega_0^\mu \, |\, u_0^\mu(t,x)> 0\},
 \]
 \[
 \Gamma_{0}^\mu=\bigcup_{a> a_0} \displaystyle\widetilde \Omega_a^\mu  . 
 \]
 One can verify that
 \[
 \widetilde \Omega_0^\mu=\widetilde\Omega_\mu,\quad 
    \Omega_0^\mu= \Omega_\mu,\quad 
    \Gamma_0^\mu=\Gamma_\mu,\quad u_0^\mu(t,x)=u_\mu(t,x)-\mu.
 \]
 Then, Lemma \ref{lem:ua-cont} implies that for almost every $\mu_0$ we have
 \[
 \forall (t,x)\in \Gamma_0^{\mu_0}\cup \Omega_{0}^{\mu_0},\qquad \lim_{\mu\downarrow \mu_0} u_0^\mu(t,x)=u_0^{\mu_0}(t,x).
 \]
 Let us also define
 \[
N^{K,A,\mu}_t:=\langle \widetilde{\Theta}^{K,\mu}_t,\mathbbm{1}_{A_t}\rangle,
\]
where $\widetilde{\Theta}_t^{K,\mu}$ is the historical birth-death process defined in \eqref{histo}  and $A_t$ defined in \eqref{At}, such that $Z_0^{K,\mu}$ is a Poisson point measure in $\R$ with intensity measure
$K^{\beta_0^K(x)-\mu}dx$. We finally deduce from Theorem~\ref{cor:HJ} that for a.e. $\mu$ and all $(t,x)\in \Gamma_0^{\mu}\cup \Omega_{0}^{\mu}$, and in particular for a.e. $(t,x)\in \R^+\times\R$,  in probability
 \beq\label{eq:perturbation}
    \lim_{\delta\to 0}\lim_{K\to\infty}\frac{1}{\log K}\log N^{K,A^{x,\delta}_t,\mu}_t=u_0^\mu(t,x). 
    \eeq \hfill $\square$
 \end{rem}

Theorems \ref{thm:bornesup} and \ref{thm:borneinf} also allow us to identify the typical lineages of the population. 
As we saw in Theorem \ref{cor:HJ}, the density of individuals of trait $x$ at time $t$ can be estimated as in \eqref{est-density-HJ} with
\[
u_0(t,x)=  \sup\, \{F_t(f); f\in \mathrm{AC}[0,t],\, f(t)=x,\, \forall s\in [0,t], \, F_s(f)\geq 0\}.
\]
Let us assume that $f_o$ is an optimal trajectory where the supremum above is attained. It is proved in Lemma \ref{lem:optimal-traj} that such a trajectory necessarily exists. Assume also that $F_s(f_o)>0$ for all $s\in [0,t]$ so that $u_0(t,x)= \lim_{a\downarrow 0}u_a(t,x)$. Then, Theorems \ref{thm:bornesup} and \ref{thm:borneinf} imply that a subpopulation of comparable size (in the logarithmic scale) to the population size close to $(t,x)$ has indeed followed the trajectory $f_o$. More precisely, let us define 
$$
A_{\delta,f_o}=\{\gamma\in \mathbb{D}[0,t]; d_{\text{Sko}}\big(\gamma,f_o)\leq \delta\},
$$
$$
G_{\delta,f_o}=\{\gamma\in \mathbb{D}[0,t]; \dSko(\gamma,f_o)< \delta\}.
$$
Then,  Theorems \ref{thm:bornesup} and \ref{thm:borneinf} imply that,  
in probability 
\[
 \limsup_{K\rightarrow +\infty}   \frac{1}{\log K}\log N^{K,A_{\delta,f_o}}_{t} \leq \sup\{F_t(f); f\in A_{\delta,f_o}, \; \forall s\in [0,t],\; F_s(f)\geq 0\}
\]
and
\[
 \liminf_{K\rightarrow +\infty}   \frac{1}{\log K}\log N^{K,G_{\delta,f_o}}_{t} \geq \sup\{F_t(f); f\in G_{\delta,f_o}, \; \forall s\in [0,t],\; F_s(f)> 0\}. 
\]
Letting $\delta\to 0$ we then obtain, in probability
\[
 u_0(t,x) \leq \liminf_{\delta\to 0}\liminf_{K\rightarrow +\infty}   \frac{1}{\log K}\log N^{K,G_{\delta,f_o}}_{t} \leq   \limsup_{\delta\to 0}\limsup_{K\rightarrow +\infty}\frac{1}{\log K}\log N^{K,A_{\delta,f_o}}_{t}\leq  u_0(t,x).
\]
where we have used  $u_0(t,x)=F_t^{f_o}$. We hence deduce the following result.
\begin{cor}
\label{cor:lineage}
Let $(t,x)\in (0,+\infty)\times \R$ and $f_o$ be an optimal trajectory such that $f_o(t)=x$ and $u_0(t,x)=F_t(f_o)$. Assume also that $F_s(f_o)>0$ for all $s\in [0,t]$. We then have in probability 
\[
 \lim_{\delta\to 0}\lim_{K\rightarrow +\infty}   \frac{1}{\log K}\log N^{K,G_{\delta,f_o}}_{t} =  \lim_{\delta\to 0}\lim_{K\rightarrow +\infty}\frac{1}{\log K}\log N^{K,A_{\delta,f_o}}_{t}=  u_0(t,x).
\]
\end{cor}
Thus any optimal trajectory $f_o$ can be interpreted as the ancestral lineage of a large part of the population having trait $x$ at time $t$.\\

Let us end this section with some comments. As explained before, the quantity 
\[\beta^{K,A}_t=\frac{1}{\log K} \log N^{K,A}_t\]
is the exponent in $K$ of the number of particles $N^{K,A}_t$, in the sense that $N^{K,A}_t=K^{\beta^{K,A}_t}$. The evolution in time of the number of particles around the path $f$ is approximately $K^{F_t(f)}$. This says that the exponent, starting from $\beta_0(f(0))$, changes according to the births and deaths along this path as $\int_0^t R(f_s)\ ds$. For comparison, remember that for a branching process without mutation and with constant growth rate $R$, $\E(N^{K,A}_t)=\E(N^{K,A}_0)e^{R t}$. The penalization by $-I_t(f)$, as we will see later, comes from the fact that the probability for an ancestral lineage to be around $f$ is of the order of $K^{-I_t(f)}$. The state constraint boundary condition tells that only paths $f$ such that $F_s(f)>0$
for all $s\in [0,t]$ are admissible: the population gets extinct on the way otherwise. Notice also that the assumption that $F_s(f)>0$ for all $s\in [0,t]$ is not contradictory with the fact that for some $s$ we may have $R(f(s))<0$.

\subsection{Comparison with previous works}
\label{sec:biblio}

Our main results take the form of large deviation estimates on population sizes. This question for spatial branching Markov processes for large time goes back to~\cite{biggins1995} and has been studied by several authors, notably~\cite{berestyckibrunetharrisharrisroberts}, from which a large part of our work is inspired. These works deal with branching Brownian motions on the line and aim at describing the particles that constitute the right front.
The method of~\cite{berestyckibrunetharrisharrisroberts} is based on additive martingales for the branching process and the spine decomposition of~\cite{hardyharris,marguet}. We use a similar approach to obtain large deviations upper bounds on the branching population size. Many other works dealing more specifically with estimates on the position of the rightmost particle in branching Brownian motion used methods based on moment estimates~\cite{fangzeitouni,zeitouni}. Similar questions for general branching random walks are studied in \cite{mallein} and are thus closer to our work. Our proof of the large deviations lower bound on the branching population size is inspired from these works, although we focus on a different question dealing with the study of local population densities in a model with inhomogeneous space dependence of rates. In \cite{maillardraoultourniaire,desmaraisschertzertalyigas}, spatial birth-death processes with interaction are used to study the effect of population sizes and demographic stochasticity on the speed of invasion fronts under different scalings. {In particular, in \cite[section 1.2]{maillardraoultourniaire} a conjecture has been provided, involving a variational formulation closely related to what we obtain at the limit in Theorems \ref{thm:bornesup} and \ref{thm:borneinf}. This variational problem features a different Lagrangian, which requires a different scaling of the problem.  }\\

  Hamilton-Jacobi equations have been widely used in  the asymptotic analysis of integro-differential equations in  evolutionary biology (see for instance \cite{BP.GB:08,Barles2009,LMP}) but also in the study of propagation phenomena (see e.g. \cite{MF:85,LE.PS:89,GB.LE.PS:90}). Let us consider the following model
 \begin{equation}
     \label{eq:Integ-Diff}
     \begin{cases}
     \e \partial_t n_\e(t,x)=\int_\R p(x+ h)n_\e(x+ h)G(h/\e)dh/\e+n_\e(b(x)-d(x)),\\
     n_\e(0,x)=\exp(\beta_0(x)/\e).
     \end{cases}
 \end{equation}
 Here $n_\e(t,x)$ stands for the phenotypic density of a population, with $t\in \R^+$ and $x\in \R$ corresponding respectively to time and to a phenotypic trait. Similarly to above, $b(x)$ and $p(x)$ stand for birth rates without and with mutations and $d(x)$ corresponds to a death rate. The mutations are distributed as $\f{1}{\e}G(\f{y}{\e})dy$. The  mutational  variance scales as $\e^2$, which is assumed to be a small parameter. A change of variable in time $t\to t/\e$ has then been taken into account to accelerate the slow dynamics   resulting from small effects of   mutations. This change of variable leads to the $\e$ coefficient in front of $\p_t n_\e$.

 Such an equation can be related to the stochastic process $Z_t^K$ introduced in Section \ref{sec:model} in two ways. Firstly, it can be obtained as a large population limit, that is $K\to +\infty$ of the stochastic process $Z_t^K/K$ (see \cite{fourniermeleard,champagnatferrieremeleard}), but taking the mutational variance constant of order $\e^2$ and independent of $K$. Secondly, the expectation of the stochastic process above satisfies \fer{eq:Integ-Diff} with $1/\log K=\e$.

The asymptotic behavior of $n_\e$ as $\e\to 0$ can be described via a Hopf-Cole transformation:
\[
U_\e(t,x)=\e \log(n_\e(t,x)).
\]
Notice the analogy of this transformation with the transformation $\f{1}{\log K}\log N_t^{K,A}$ used above.
It is proved in \cite{Barles2009} (in a slightly different setting, taking into account a competition term) that as $\e\to 0$, $U_\e$ converges to the unique viscosity solution of
\begin{equation}
\label{HJ:whole-domain}
    \begin{cases}
    \p_t U=p(x)H(\p_x U)+ R(x),\qquad x\in \R\\
    U(0,x)=\beta_0(x),
\end{cases}
\end{equation}
with $H$ and $R$ defined in \eqref{eq:def-H} and \fer{def:R}. Notice that this is the same equation as \fer{eq:HJ} but set in the whole domain $\R$. When considering the asymptotic behavior of the stochastic process instead of the deterministic integro-differential equation, possible extinction of small subpopulations are taken into account. This leads to a smaller limit $u_0\leq U$. The limit $u_0$ of the stochastic process can in particular take $-\infty$ as value. The variational formulation of the problem given in \eqref{def:u-a} provides  an intuitive explanation. The maximal trajectories in the variational problem correspond indeed to the typical trajectories of lineages as obtained in Theorems \ref{thm:bornesup} and \ref{thm:borneinf}. If the value function on such a trajectory takes a negative value $-c<0$, then the expected population size approaches the small size of order $K^{-c}$ as proved later in Theorem~\ref{thm27bis}, which results with high probability in extinction,  so the trajectory should be excluded. This is a significant difference with the deterministic derivation where all trajectories are allowed, leading to the Hamilton-Jacobi equation \eqref{HJ:whole-domain} in the whole domain. In the stochastic derivation the limit $u_0$ is positive in the set $\overline \Omega_0$ and equal to $-\infty$ outside this set. The function $u_0$ satisfies both a Dirichlet boundary condition and a state constraint condition. Several comments are in order.

(i) In deterministic works,  one usually considers a slightly more complex model taking into account a nonlocal mortality rate due to competition \cite{BP.GB:08,Barles2009,LMP}. In this case, the growth rate is given by $R(x,I_\e(t))$, with $I_\e(t)=\int_\R n_\e(t,x)dx$. Such a mortality rate leads often to a constraint of type
\[
\max_{x\in \R} U(t,x)=0.
\]
This constraint on the limit $U$ might seem confusing when it is combined with the threshold on the trajectories in the stochastic derivation leading to $u_0\geq 0$. Note however that in order to make a relevant comparison between these results we have to divide the population process $Z_t^K$ by $K$ (see \cite{champagnatferrieremeleard} where \eqref{eq:Integ-Diff} has been derived from a stochastic model). This means that, when put in a similar framework than usual deterministic works, one has to put the threshold of extinction equal to $-1$ instead of $0$. The expected equation, in presence of a competition term, would then be given by
\[ 
 \begin{cases}
 \p_t u=H(x,\p_x u)+R(x,I),& (t,x)\in \Omega^I\\
  \max_{x\in \R}u(t,x)=0,\\
 u(t,x)=-1,& (t,x)\in \p  \Omega^I , \, t>0,\\
 u(0,x)=\beta_0(x)-1, &\text{for all $x$ s.t $(0,x)\in \Omega^I$},
 \end{cases}
\]
with a set $\Omega^I$ which depends on the competition term $I(\cdot)$ and would be such that $u$ also satisfies a state constraint boundary condition in $\Omega^I$. Obtaining this equation rigorously is the aim of a future work.

(ii) Biological criticisms were made on the Hamilton-Jacobi method because of the so-called tail problem  \cite{perthame2010}. Artifacts may indeed arise due to an inadequate treatment of small subpopulations. Exponentially small subpopulations which actually may be extinct can have a strong influence on the future of the population. Artificial jumps of the dominant trait may occur. The branching patterns are also too fast. Modifications of the Hamilton-Jacobi equation were suggested in \cite{perthame2010,jabin2012,Mirrahimi2012} to solve this problem. These modifications were directly made to deterministic models. Here we use a stochastic individual-based approach providing a more biologically relevant justification of the outcome. Note however that we obtain a {closely related limit to} \cite{Mirrahimi2012}, even though this correction was described in a less direct and less precise way in \cite{Mirrahimi2012}. Moreover the threshold of extinction that was considered in \cite{Mirrahimi2012} was arbitrary. Here the threshold is obtained with a direct link to the population size.

(iii) As mentioned above, Hamilton-Jacobi equations are widely used to provide approximations of the phenotypic density of a population in a small mutational variance regime \cite{BP.GB:08,Barles2009,LMP}. Here, we go further than characterizing the phenotypic density. We also identify the typical lineages of the population thanks to Corollary \ref{cor:lineage}. A previous work \cite{foriengarnierpatout} had already made a link between the optimal trajectories of the Hamilton-Jacobi equation and the typical lineages of the population. These authors considered a deterministic model in a context of changing environment and they used the neutral fractions approach to study the inside dynamics of the population. Probabilist spinal techniques and historical birth and death processes that allow to link the typical lineages with stochastic individual-based models have been used in \cite{CHMT,HenryMeleardTran} to describe the phylogenies but only under a large population limit, whereas here, mutations and time are also rescaled. Some techniques that are used in the present work, such as using Feynman-Kac formulas for characterizing spines, are still taken from these papers (see also \cite{hardyharris}).

\section{Study of an auxiliary jump process}
\label{sec:LDP}

\subsection{A Feynman-Kac formula for $\E (N^{K,A}_t)$}

Given $t>0$ and $A\subset\mathbb{D}[0,t]$, we interpret $\E( N^{K,A}_t)$ as the expectation of a functional of an auxiliary process based on the mutations dynamics.

Based on \eqref{semimartingale}, let us introduce a random walk in continuous time $(X^K_t)_{t\in \R_+}$ with paths in $\mathbb{D}$ and infinitesimal generator 
\begin{equation}
    \label{eq:gene-X^K}
    \mathcal{L}^K\varphi(x)
   = p(x)\int_\R\left[\varphi\left(x+\frac{y}{\log K}\right)-\varphi(x)\right]G(y)dy,
\end{equation}
 defined for any measurable bounded function $\varphi$.

\medskip 
A pathwise representation similar to~\eqref{eq:eds_Poisson} of the process $(X^K_t)_{t\geq 0}$ can be obtained as follows: let us give ourselves (on some probability space) a Poisson point measure $Q(ds,d\theta,dy)$ on $\R_+\times \R_+\times \R$ with intensity $G(y)dy \otimes d\theta \otimes ds$ and an independent real random variable $X^K_0$. We can write
\begin{equation}
\label{marchealeatoire}
   X^K_t =  X^K_0 + \int_0^t \int_{\R_+}\int_{\R} \frac{y}{\log K}\ind_{\{\theta\le p(X^K_{s-}) \}}Q(ds,d\theta,dy). 
\end{equation}

Let us also define 
\begin{equation}
\label{eq:def-Y}
    Y^K_t = X^K_{t\log K},
\end{equation}
the process in the time scale $\log K$. 
For all $t>0$ and $x\in\R$, we denote by $\mu^K_{x,t}$ the law of the process $(Y^K_{s})_{s\in [0,t]}$ on $\mathbb{D}[0,t]$ conditionally on $Y^K_0=X^K_0=x$, and by $\E_{\mu^K_{x,t}}$ the corresponding expectation. 

\medskip 
We have the following classical Feynman-Kac representation of $N^{K,A}_t$ (also called many-to-one formula).

\begin{prop}\label{prop:FeynmanKac-bis}
Let $x\in\R$. We consider the birth-death-mutation process $(Z^K_t,t\ge 0)$ defined as before but started from a unique individual with trait $x$, i.e.\ $Z^K_0=\delta_{x}$, and we denote the corresponding expectation by $\mathbb{E}_{\delta_x}$.\\ 
(i) Let $\varphi:\R\to\R$ be bounded and measurable. Then, for any $t>0$, we have
		\begin{equation}
			\label{eq:MTO}
			\mathbb{E}_{\delta_{x}}\left[\langle {Z}^K_{t}, \varphi\rangle \right]=\mathbb{E}_{x}\left[\exp\left(\int_{0}^{t} R(X^K_s)\ ds \right)\varphi(X^K_{t}) \right],
		\end{equation}
		where $X^K$ is the  process defined in  \eqref{marchealeatoire}.\\
(ii) For $t>0$, $x\in \R$ and for a bounded measurable function $\Phi:\mathbb{D}[0,t]\to\mathbb{R}$,
		\begin{equation}
		 			\label{eq:MTO099}
			\mathbb{E}_{\delta_x}\left[\langle {\Theta}^K_t,\Phi\rangle\right]=\mathbb{E}_{\delta_{x}}\left[\sum_{u\in {V}^K_{t}}\Phi({X}^{K,u}_{s},\ s\leq t) \right]
            =\mathbb{E}_{x}\left[\exp\left(\int_{0}^{t} R(X^K_{s})\ ds \right)\Phi(X^K_{s},\ s\leq t) \right].   
		\end{equation}
(iii) For all $t>0$ and $A\subset\mathbb{D}[0,t]$,
\begin{align*}
        \E (N^{K,A}_{t})  & =\int_\R K^{\beta^K_0(x)}\E_{\mu^K_{x,T}}\Big[\exp\Big(\log K\int_0^{t} R(Y_s^K) ds \Big)\mathbbm{1}_{(Y^K_{s})_{s\in[0,t]}\in A}\Big]dx.
\end{align*}
\end{prop}

\begin{proof}The proof of point (i) for $\varphi\in C_b(\R)$ is given in Section \ref{app:manyto1} for the sake of completeness. Since the set of bounded continuous functions on $\R$ is dense for the bounded pointwise topology in the set of bounded measurable functions~\cite[Prop.\ 4.2, Chap.\ 3]{ethierkurtz}, Point~(i) extends easily to any $\varphi$ bounded measurable. Note that the Feynman-Kac formula \eqref{eq:MTO} concerns the law of first moments of $Z^K_{t}$ (with fixed $t$) issued from one individual with trait $x$ and can be extended to the whole trajectory using standard techniques (see~\cite{CHMT, marguet}), providing (ii). Point (iii) then follows from \eqref{conditioninitiale} changing time $t$ by $t\log K$ by change of variables. 
\end{proof}

In what follows, we will also need a many-to-one formula for the whole tree, as in \cite{bansayedelmasmarsalletran,marguet}. Recall that $V^K_{[0,t]}=\bigcup_{s\in[0,t]}V^K_s$ is the set of individuals born before time $t$ (including those still alive at time $t$). Recall that our continuous-time birth-death process is associated with a binary tree where each node corresponds to a birth or death event. If it is a death event, the node is a leaf. If it is a birth event, then the individual $v$ is replaced with $v0$ and $v1$, where $v0$ is the continuation of the mother $v$ (with the same trait) and $v1$ is the new offspring (with a possible mutated trait). Let us also denote by $S^0_v$ the birth time of $v$ and by $S_v$ the time at which $v$ disappears (either by death or reproduction). 

\begin{prop}\label{prop:forks}
   	We have that for all $t>0$, $\Phi :\mathbb{D}[0,t]\times[0,t]\to\mathbb{R}$ a bounded measurable function and $x\in \mathbb{R}$:
		\begin{multline}
			\label{eq:forks}
			\mathbb{E}_{\delta_{x}}\left[\sum_{v\in {V}^K_{[0,t]}} \Phi\big(({X}^{K,v}_{r \wedge S_v},\ r\leq t),S_v \wedge t\big) \right]\\
            =\int_0^t  \E_x\Bigg[\Phi\big((X_{r\wedge s}^{K}, r\leq t),s\big) (b+p+d)(X^K_s) \exp\Big(\int_0^s R(X^K_r)dr\Big)\Bigg]\ ds\\
        + \E_x\Bigg[\Phi\big((X_r^{K}, r\leq t),t\big)  \exp\Big(\int_0^t R(X^K_r)dr\Big)\Bigg].
		\end{multline} 
        Notice that if there exists a function $\Psi\ :\ \D[0,t]\times [0,t]\rightarrow \R$ such that for all $(f,s)\in \D[0,t]\times [0,t]$,
$\Phi(f, s)= \Psi\big(f,s)\ind_{s<t}$, 
then the second term in the right hand side of \eqref{eq:forks} vanishes. 
\end{prop}

The proof of Proposition \ref{prop:forks} is deferred to Appendix \ref{app:prop-forks}.

Using Proposition \ref{prop:forks}, we can establish an identity for forks, i.e. sums over pairs of individuals living at time $t$ (also called many-to-two formula). We do not state this corollary in full generality as this would require extra notations, but will prove a version tailored for our needs in Section \ref{sec:estimateNfeps} (see Lemma~\ref{lem-appli-fourches}).

\subsection{Large deviation principle for $(\mu^K_{x,t})_{K\geq 1}$}

For all $t>0$ and $x\in\R$, we define the function $I_t$ restricted to functions starting from $x$ as follows: for all $f\in\mathbb{D}[0,t]$,
\[
I_{t,x}(f)=\begin{cases}
    I_t(f) & \text{if }f(0)=x, \\ +\infty & \text{otherwise.}
\end{cases}
\]
Note that $I_t(f)=I_{t,f(0)}(f)$.\\
The following large deviations principle is a direct application of~\cite[Theorem 10.2.6]{dupuisellis}. Indeed, the conditions 10.2.2 and 10.2.4 in this theorem are obviously satisfied thanks to the assumptions on the measure $p(x)G(y)dy$.

\begin{theorem}
    \label{thm:main-PGD-bis}
    For all $t>0$ and $x\in \R$, the family of laws $(\mu^K_{x,t})_{K\geq 1}$ satisfies a large deviation principle on $\mathbb{D}[0,t]$ with rate $1/\log K$ and good rate function $I_{t,x}$: for any subset $A\subset \mathbb{D}[0,t]$,
\begin{multline}
    -\inf_{f\in \mathrm{int}(A)} I_{t,x}(f)\leq \liminf_{K\rightarrow +\infty} \frac{1}{\log K} \mathbb{P}_x\big((X^K_{s\log K})_{s\leq t} \in A\big)\\
    \leq \limsup_{K\rightarrow +\infty} \frac{1}{\log K} \mathbb{P}_x\big((X^K_{s\log K})_{s\leq t} \in A\big) \leq -\inf_{f\in \mathrm{cl}(A)} I_{t,x}(f),
\end{multline}where $\mathrm{int}(A)$ and $\mathrm{cl}(A)$ are the interior and closure of $A$ for the Skorohod topology. \\
    In addition, the family of measures $(\mu^K_{x,t})_{K\geq 1}$ satisfies the Laplace principle with rate $1/\log K$ and rate function $I_{t,x}$ uniformly on compact sets with respect to $x\in\R$, as recalled below. 
\end{theorem}

We recall that the rate function $I_{t,x}$ is good if it is lower semi-continuous and if $\{I_{t,x}\leq M\}$ is compact for all $M<\infty$. We also recall that $\mu^K_{x,t}$ satisfies the Laplace principle with rate $1/\log K$ and rate function $I_{t,x}$ uniformly on compact sets with respect to $x\in\R$ if, for all compact subset $A$ of $\R$ and all constant $M<\infty$,
\begin{equation}
\label{eq:bft-unif-compact}
   \bigcup_{x\in A}\{f\in\mathbb{D}[0,t],\ I_{t,x}(f)\leq M\} 
\end{equation}
is compact and for all bounded continuous function $\phi:\mathbb{D}[0,t]\to \R$
\begin{equation}
    \label{eq:bft-unif-compact-2}
    \lim_{K\to\infty}\sup_{x\in A}\left|\frac{1}{\log K}\log \E_{\mu^K_{t,x}}\left(K^{-\phi(Y^K_s,s\leq t)}\right)+\inf_{f\in\mathbb{D}[0,t]}(\phi(f)+I_{t,x}(f))\right|=0.
\end{equation}See \cite[Def.\,1.2.6]{dupuisellis}.

\medskip

Let us prove a uniform in $x$ version of the exponential tightness of the measures $\mu^K_{x,t}$.

\begin{lemma}
\label{lem:expo-tight}
    For all $t>0$, the family of measures $(\mu^K_{x,t})_{K,x}$ is exponentially tight, uniformly on compact sets. This means that, for all $M<\infty$ and all compact subset $B$ of $\R$, there exists a compact subset $A$ of $\mathbb{D}[0,t]$ such that
    \[
    \limsup_{K\to\infty}\sup_{x\in B}\frac{1}{\log K}\log\mu^K_{x,t}(A^c)\leq -M.
    \]
\end{lemma}

\begin{proof}
The proof of Lemma~\ref{lem:expo-tight} is inspired form~\cite[Exercice 4.1.10]{dembozeitouni}. 
\medskip

\noindent\textit{Step 1: asymptotic bound.} Our first goal is to prove that, for all $M_0<\infty$, all $\eta>0$ and all compact subset $B$ of $\R$, there exist an integer $m\geq 1$, functions $f_1,
\ldots, f_m\in\mathbb{D}[0,t]$ and an integer $K_0$ such that, for all $K\geq K_0$,
\begin{equation}
\label{eq:step1-lem-expo-tight}
    \sup_{x\in B}\mu^K_{x,t}\left[\left(\bigcup_{i=1}^m B_\eta(f_i)\right)^c\right]\leq K^{-M_0},
\end{equation}
where $B_{\eta}(f)$ is the open ball of radius $\eta$ centered at $f$ for Skorohod's distance $d_\text{Sko}$.

    Fix $\eta>0$, $M_0<\infty$ and $B\subset\R$ compact. The set $\mathbb{D}[0,t]$ equipped with the Skorohod topology is a Polish space, so in particular there exists a dense sequence $(f_n)_{n\geq 1}$. Let $m\in\mathbb{N}$. For all $k\geq 1$, we define for $f\in \mathbb{D}[0,t]$
    \begin{equation}
    \label{eq:def-phi_k}
    \phi_k(f)=k\,d_\text{Sko}\left(f,(\cup_{i=1}^m B_{\eta}(f_i))^c\right).        
    \end{equation}
    Since the function $\phi_k$ is bounded and continuous, it follows from the Laplace principle uniformly on compact sets~\eqref{eq:bft-unif-compact-2} that
    \begin{align}
    \limsup_{K\to\infty}\sup_{x\in B}\frac{1}{\log K}\log \mu^K_{x,t}\left[\left(\bigcup_{i=1}^m B_{\eta}(f_i)\right)^c\right] & \leq 
            \limsup_{K\to\infty}\sup_{x\in B}\frac{1}{\log K}\log \E_{\mu^K_{x,t}}(K^{-\phi_k}) \\ 
            & \leq -\inf_{x\in B}\inf_{f\in\mathbb{D}[0,t]}(\phi_k(f)+I_{t,x}(f)) .
            \label{eq:LDP-unif}
    \end{align}
    If $f\in \cup_{i=1}^m B_{\eta/2}(f_i)$, $\phi_k(f)\geq k\eta/2$, hence
    \[
    \limsup_{K\to\infty}\sup_{x\in B}\frac{1}{\log K}\log \mu^K_{x,t}\left[\left(\bigcup_{i=1}^m B_{\eta}(f_i)\right)^c\right]\leq -\left(\frac{k\eta}{2}\wedge \inf_{x\in B}\ \inf_{f\not\in\cup_{i=1}^m B_{\eta/2}(f_i)}I_{t,x}(f)\right).
    \]
    Now,~\eqref{eq:bft-unif-compact} implies that
    \[
    \bigcup_{x\in B}\left\{f\in\mathbb{D}[0,t],\, I_{t,x}(f)\leq M_0+1\right\}
    \]
    is compact, hence there exists $m\geq 1$ such that this set is included in $\cup_{i=1}^m B_{\eta/2}(f_i)$. Therefore, for such a value of $m$,
    \[
    \limsup_{K\to\infty}\sup_{x\in B}\frac{1}{\log K}\log \mu^K_{x,t}\left[\left(\bigcup_{i=1}^m B_{\eta}(f_i)\right)^c\right]\leq-\frac{k\eta}{2}\wedge (M_0+1).
    \]
    Choosing $k> 2M_0/\eta$ ends the proof of~\eqref{eq:step1-lem-expo-tight}.
    \medskip

    \noindent\textit{Step 2: uniform bound.} We now prove the  following stronger version of~\eqref{eq:step1-lem-expo-tight}: for all $M_0<\infty$, $\eta>0$ and compact subset $B$ of $\R$, there exists an integer $m\geq 1$ and functions $f_1,
\ldots, f_m\in\mathbb{D}[0,t]$ such that, for all $K\geq 1$,
\begin{equation}
\label{eq:step1-lem-expo-tight-2}
    \sup_{x\in B}\mu^K_{x,t}\left[\left(\bigcup_{i=1}^m B_\eta(f_i)\right)^c\right]\leq K^{-M_0}.
\end{equation}

In view of Step 1, to prove this, increasing $m$ if necessary, it is sufficient to prove that for any $K\leq K_0-1$, for all $m$ large enough,
    \[
    \sup_{x\in B}\mu^K_{x,t}\left[\left(\bigcup_{i=1}^m B_{\eta}(f_i)\right)^c\right]\leq K^{-M_0}.
    \]
    So we fix $K\leq K_0-1$ and $M_0<\infty$ until the end of Step 2.
    The last claim follows from a continuity property of the process $Y^K$ constructed from~\eqref{marchealeatoire} and~\eqref{eq:def-Y} with respect to its initial value: given $x\in\R$ and a Poisson point measure $Q^K(du,d\theta,dy)$ on $\R_+\times \R_+\times \R$ with intensity $(\log K)dud\theta G(y)dy$, we define $Y^{K,x}$ as the solution to
    \[
       Y^K_t =  x + \int_0^t \int_{\R_+}\int_{\R} \frac{y}{\log K}\ind_{\{\theta\le p(Y^K_{u-}) \}}Q^K(du,d\theta,dy)
    \]
    and we define $\widetilde{Y}^{K,x}$ as the solution to
    \[
       \widetilde{Y}^K_t =  x + \int_0^t \int_{\R_+}\int_{\R} \frac{y}{\log K}\ind_{\{\theta< p(\widetilde{Y}^K_{u-}) \}}Q^K(du,d\theta,dy). 
    \]
    Notice that the inequality in the definition of $\widetilde{Y}^K_t$ is strict, contrary to the one in the definition of $Y^K_t$.
    By standard properties of Poisson point measures, $Y^{K,x}=\widetilde{Y}^{K,x}$ almost surely. In addition, by continuity of the function $p$, for any $\omega$ in the event $\{Y^{K,x}=\widetilde{Y}^{K,x}\}$, the map
    \[
    y\mapsto (Y^{K,y}_s(\omega)-y)_{s\leq t}
    \]
    is constant for $y$ in a neighborhood of $x$. Therefore, it follows from Lebesgue's theorem that, for any measurable set $G\subset\mathbb{D}[0,t]$, defining for all $y\in\R$
    \begin{equation}\label{notation:shift}
    G_y:=\left\{f+y,\, f\in G\right\},
    \end{equation}
    the map
    \[
    y\mapsto \mu^K_{y,t}(G_y)
    \]
    is continuous at $x$. Since $x$ was arbitrary, we deduce that this map is continuous. More precisely, for all $x\in\R$ and $G\subset\mathbb{D}[0,t]$,
    \[
    \left|\mu^K_{y,t}(G_y)-\mu^K_{x,t}(G_x)\right|\leq \P\left(Y^{K,y}_s-y\neq Y^{K,x}_s-x,\text{ for some } s\in[0,t]\right)\xrightarrow[y\to x]{}0.
    \]
    In particular, there exists $\delta_x>0$ such that for all $y\in[x-\delta_x,x+\delta_x]$,
    \begin{equation}
        \label{eq:blabla}
    \P\left(Y^{K,y}_s-y\neq Y^{K,x}_s-x,\text{ for some } s\in[0,t]\right)\leq\frac{K^{-M_0}}{2}.
     \end{equation}
    We can assume without loss of generality that $\delta_x\leq\eta/2$ for all $x\in \R$. Since the compact set $B$ is included in the union of the intervals $(x-\delta_x,x+\delta_x)$ there exist $N<\infty$ and $x_1,\ldots,x_N\in\R$ such that
    \[
    B\subset\bigcup_{j=1}^N (x_j-\delta_{x_j},x_j+\delta_{x_j}).
    \]
    For any $j\in\{1,\ldots,N\}$, since $(f_i)_{i\geq 1}$ is dense in $\mathbb{D}[0,t]$, there exists $m_j$ large enough such that
    \[
    \mu^K_{x_j,t}\left[\left(\bigcup_{i=1}^{m_j}B_{\eta/2}(f_i)\right)^c\right]\leq\frac{K^{-M_0}}{2}.
    \]
    Our goal is   now to extend this estimate to any $x\in B$. Let $x\in B$, there exists $j\in\{1,\ldots,N\}$ such that $|x-x_j|\leq\delta_{x_j}$. Recalling that $\delta_{x_j}<\eta/2$,  for all $m\geq \sup_{1\leq j\leq N} m_j$, we have the inclusion,
    \[
    \left(\bigcup_{i=1}^{m}B_{\eta/2}(f_i)\right)_{x-x_j}\subset \left(\bigcup_{i=1}^{m}B_\eta(f_i)\right).
    \]where the notation $()_{x-x_j}$ has been defined in \eqref{notation:shift}.
    Therefore,
    \begin{align*}
            \mu^K_{x,t}\left[\left(\bigcup_{i=1}^{m}B_\eta(f_i)\right)^c\right] & \leq \mu^K_{x,t}\left[\left(\bigcup_{i=1}^{m}B_{\eta/2}(f_i)\right)_{x-x_j}^c\right] \\ & \leq \left|\mu^K_{x,t}\left[\left(\bigcup_{i=1}^{m}B_{\eta/2}(f_i)\right)_{x-x_j}^c\right]-\mu^K_{x_j,t}\left[\left(\bigcup_{i=1}^{m}B_{\eta/2}(f_i)\right)^c\right]
            \right| \\ & \quad+\mu^K_{x_j,t}\left[\left(\bigcup_{i=1}^{m_j}B_{\eta/2}(f_i)\right)^c\right] \\ & \leq \P\left(Y^{K,y}_s-y\neq Y^{K,x}_s-x,\text{ for some } s\in[0,t]\right)+\frac{K^{-M_0}}{2}. 
    \end{align*}
    Hence, by~\eqref{eq:blabla}, we have proved~\eqref{eq:step1-lem-expo-tight-2} for $m\geq \sup_{1\leq j\leq N} m_j$. 
    \medskip

    \noindent\textit{Step 3: Conclusion of the proof.} We fix $M<\infty$ and a compact subset $B$ of $\R$. For all integer $k\geq 1$, we apply~\eqref{eq:step1-lem-expo-tight-2} with $\eta= 1/k$ and $M_0=k M$: there exists $m_k<\infty$ such that for all $K\geq 1$
    \[
        \sup_{x\in B}\mu^K_{x,t}\left[\left(\bigcup_{i=1}^{m_k} B_{1/k}(f_i)\right)^c\right]\leq K^{-kM}.
    \]
    This implies that
    \[
    \sup_{x\in B}\mu^K_{x,t}\left[\left(\bigcap_{k=1}^\infty\ \bigcup_{i=1}^{m_k} B_{1/k}(f_i)\right)^c\right]\leq \frac{K^{-M}}{1-K^{-M}}.
    \]
    Now, the set 
    \[
    \bigcap_{k=1}^\infty\ \bigcup_{i=1}^{m_k} B_{1/k}(f_i)
    \]
    is precompact. Indeed, considering any sequence $(\phi_j)_{j\geq 1}$ in this set, since it belongs to $\cup_{i=1}^{m_1} B_{1}(f_i)$, there exists $i_1\in\{1,\ldots,m_1\}$ such that $\phi_j\in B_{1}(f_{i_1})$ infinitely often. Using the diagonal extraction procedure, we deduce that there exists a subsequence, still denoted $(\phi_j)_{j\geq 1}$ for convenience, and integers $i_k\in\{1,\ldots,m_k\}$ such that $\phi_j\in B_{1/k}(f_{i_k})$ for all $j$ large enough. This implies that the sequence $(\phi_j)_{j\geq 1}$ is Cauchy, hence the conclusion.

    Therefore, the compact set
    \[
    \overline{\bigcap_{k=1}^\infty\ \bigcup_{i=1}^{m_k} B_{1/k}(f_i)}
    \]
    satisfies the claim of Lemma~\ref{lem:expo-tight}.
\end{proof}

\subsection{Non-variational form and  domain of the rate function $I_t$ }\label{sec:I-nonvariational}

Note that the function $H$ is a convex function and $H'$ is a $C^1$-diffeomorphism from $\mathbb{R}$ to itself. We have the following result, providing an alternative, non-variational expression for the rate function $I_t$ and characterizing the set of functions $f$ such that $I_t(f)<+\infty$.

\begin{lemma}
\label{lem:fct-tx}
    (i) For all $t$ and $f\in C^2[0,t]$,
    \begin{align}
        I_{t}(f)
        = & \psi_f(t)f(t)-\psi_f(0)f(0)-\int_0^t \left(f_s\psi'_f(s)+p(f_{s})H(\psi_f(s))\right)ds,\label{eq:lem-LDP-LB}
    \end{align}
    where for all $s\in[0,t]$, \[\psi_f(s)=(H')^{-1}\left(\frac{\dot{f}_s}{p(f_s)}\right).\]

    (ii) For all $f\in AC[0,t]$, $I_t(f)<+\infty$ iff
    \[
    \int_0^t\dot{f}_s(H')^{-1}\left(\frac{\dot{f}_s}{\bar{p}}\right)ds=\int_0^t|\dot{f}_s|(H')^{-1}\left(\frac{|\dot{f}_s|}{\bar{p}}\right)ds<+\infty.
    \]
\end{lemma}

\begin{proof}
Let $f\in C^2[0,t]$. From \eqref{eq:def-L} and \eqref{def=fonct-action}, we have that 
\[
I_{t}(f)= \int_0 ^t \sup_{\alpha\in \R} \big(\alpha \dot{f}_s- p(f_s)H(\alpha) \big) ds.
\]
For a given $s\in [0,t]$, the optimization problem $\sup_{\alpha\in \R} \big(\alpha \dot{f}_s- p(f_s)H(\alpha)\big)$ reaches its supremum for $\alpha$ that solves
$H'(\alpha)=\dot{f}_s/p(f_s)$,
i.e.\ for $\alpha=\psi_f(s)$. Hence
\[I_t(f)=\int_0^t \Big(\psi_f(s) \dot{f}_s - p(f_s)H\big(\psi_f(s)\big)\Big) \ ds.\]
For $f\in C^2[0,t]$, the function $\psi_f$ is $C^1$, hence integration by parts yields \eqref{eq:lem-LDP-LB}.

We now proceed with the proof of (ii).
Observe that $L(x,\beta)=p(x)\widetilde L(\beta/p(x))$, with $\widetilde L(v):=\sup_{\alpha\in\R}\{\alpha v-H(\alpha)\}$ convex, non-decreasing on $[0,+\infty)$ and non-increasing on $(-\infty,0]$ with $\widetilde L(0)=0$. Since the optimization problem in the definition of $\widetilde L$ has a unique solution, we obtain $\widetilde L(v)=v(H')^{-1}(v)-H(H'^{-1}(v))$. \\
    We next deduce from the change of variable $y\to -y$ and from the fact that $G$ is symmetric that for all $\alpha\in\R$
    \[
    H'(\alpha)=\int_\R\frac{y e^{\alpha y}-ye^{-\alpha y}}{2}G(y)dy
    \]
    and
    \[
    H''(\alpha)=\int_\R\frac{y^2 e^{\alpha y}+y^2e^{-\alpha y}}{2}G(y)dy.
    \]
    Hence, using that {$x\sinh x\leq x^2\cosh x$ for all $x\in\R^+$},
    \begin{equation}\label{etape_nic1}
    |H'(\alpha)|\leq\int_\R |y|\sinh(|\alpha y|)G(y)dy\leq\frac{1}{|\alpha|}\int_\R(\alpha y)^2\cosh(\alpha y)G(y)dy=|\alpha|H''(\alpha).
    \end{equation}
    This implies that for all $a>1$ and all $x\in\R$
    \begin{equation}
    \label{eq:uocuoc}
            |(H')^{-1}(x)|\leq |(H')^{-1}(ax)|\leq a|(H')^{-1}(x)|.
    \end{equation}
    Indeed, $H'$ is an increasing homeomorphism from $\R$ to itself and, since $G$ is symmetric, $H'(x)$ and $(H')^{-1}(x)$ have the same sign as $x$, so the first inequality is clear and it is enough to check the second inequality of~\eqref{eq:uocuoc} for $x>0$. In this case, we have
    \begin{align*}
        \ln (H')^{-1}(ax)-\ln (H')^{-1}(x) & =\int_1^a\frac{x((H')^{-1})'(rx)}{(H')^{-1}(rx)}dr=\int_1^a\frac{x}{(H')^{-1}(rx)H''((H')^{-1}(rx))}dr \\ & \leq\int_1^a\frac{dr}{r}=\ln a,
    \end{align*}using \eqref{etape_nic1} with $\alpha=(H')^{-1}(rx)$. Hence~\eqref{eq:uocuoc} is proved.

    Using the fact that there exists some $x_0>0$ such that, for all $x\geq x_0$ or $x\leq-x_0$,
    \[
    e^x-1\leq \frac{x}{2}e^x,
    \]
    and thus, using that $xe^x\geq -e^{-1}$ for all $x\in\mathbb{R}$, we deduce that for all $x\in\mathbb{R}$,
    \[
    e^x-1\leq \frac{x}{2}e^x+\frac{e^{-1}}{2}+e^{x_0}-1.
    \]
    This implies that, for all $\alpha\in\R$,
    \[
    H(\alpha)\leq \frac{\alpha}{2} H'(\alpha)+\frac{e^{-1}}{2}+e^{x_0}-1.
    \]
    Since $H(\alpha)\geq 0$ for all $\alpha$, we deduce that, for all $x,\beta\in\R$,
    \begin{equation}
    \label{eq:borne-L-2}
            \beta (H')^{-1}\left(\frac{\beta}{p(x)}\right)\geq L(x,\beta)\geq \frac{\beta}{2} (H')^{-1}\left(\frac{\beta}{p(x)}\right)-\bar{p}(\frac{e^{-1}}{2}+e^{x_0}-1).
    \end{equation}
    Then, it follows from~\eqref{eq:uocuoc} that
    \begin{equation*}
            |\beta|(H')^{-1}\left(\frac{|\beta|}{\bar{p}}\right)\leq |\beta|(H')^{-1}\left(\frac{|\beta|}{p(x)}\right)=\beta(H')^{-1}\left(\frac{\beta}{{p(x)}}\right)\leq\frac{\bar{p}}{\underline{p}}|\beta|(H')^{-1}\left(\frac{|\beta|}{\bar{p}}\right).
    \end{equation*}
    Therefore, Point~(ii) follows.
\end{proof}

\section{Large deviations estimates on $\mathbb{E} (N^{K,A}_t)$}
\label{sec:in-expectation}

In order to prove the large deviations estimates on $N^{K,A}_t$ of Theorems~\ref{thm:bornesup} and~\ref{thm:borneinf}, we
first need large deviations estimates on $\E (N^{K,A}_t)$.

\begin{theorem}\label{thm27bis}
For any $t>0$ and for any closed $A\subset\mathbb{D}[0,t]$, 
\begin{equation}
    \label{eq:coucou3}
  \limsup_{K\rightarrow +\infty} \frac{1}{\log K} \log\big(\mathbb{E}[N_t^{K,A}]\big)
\le  \sup_{f\in A} F_t(f),
\end{equation}
and for any open $G\subset\mathbb{D}[0,t]$
\begin{equation}
    \label{eq:coucou4}
  \liminf_{K\rightarrow +\infty} \frac{1}{\log K} \log\big(\mathbb{E}[N_t^{K,G}]\big)
\ge  \sup_{f\in G} F_t(f),
\end{equation}
     with the usual convention that $\sup\emptyset=-\infty$.
\end{theorem}

Note that, in contrast with Theorems~\ref{thm:bornesup} and~\ref{thm:borneinf}, there is no state constraint in the right-hand sides of~\eqref{eq:coucou3} and~\eqref{eq:coucou4}. This shows the fundamental difference between the behavior of the stochastic process and its expectation, already observed in~\cite{biggins1995,berestyckibrunetharrisharrisroberts}.

\begin{proof}[Proof of the upper bound in Theorem~\ref{thm27bis}]
    Let $A\subset\mathbb{D}[0,t]$ be closed. 
 Assume first that 
 \[\sup_{f\in A}F_t(f)> -\infty.\]
    Since $\beta_0(x)\to -\infty$ when $x\to\pm\infty$ and $\beta^K_0$ converges to $\beta_0$ for the uniform norm, there exists a compact subset $B$ of $\R$ such that 
    \begin{equation}
        \label{eq:borne-B}
        \forall x\in B^c,\quad \beta^K_0(x)\leq -t(\bar{b}+\bar{p})-1+\sup_{f\in A} F_t(f).
    \end{equation}
    Let $A_0$ be the compact set given by Lemma~\ref{lem:expo-tight} with this choice of $B$ and with the constant 
    \begin{equation}
        \label{eq:choix-M}
        M=\sup_{K\geq 1}\sup_{x\in\R}\beta^K_0(x)+(\bar{b}+\bar{p})t+1-\sup_{f\in A} F_t(f).
    \end{equation}  
    Fix $\varepsilon>0$ and $f\in A\cap A_0$. Since $\beta_0$, $b$, $d$ and $p$ are continuous, there exists $\delta_f>0$ such that, for all $K$ large enough, for all $g\in B_{\delta_f}(f)$,
    \[
    \beta^K_0(g(0))+\int_0^t R(g(s))ds\leq \beta_0(f(0))+\int_0^t R(f(s))ds+\varepsilon.
    \]
    In addition, since the function $I_t$ is lower semi-continuous on $\mathbb{D}[0,t]$ (cf.~\cite[Thm.\,3.2.1]{Fathi-book}), we can also assume reducing $\delta_f>0$ if necessary that, for all $x\in[f(0)-\delta_f,f(0)+\delta_f]$ and all $g\in B_{2\delta_f}(f)$,
    \begin{equation}
    \label{eq:I-sci}
            I_{t,x}(g)\geq I_{t,f(0)}(f)-\varepsilon.
    \end{equation}
    
    We introduce for all $k\geq 1$ the function $\phi_{k,f}$ on $\mathbb{D}[0,t]$ defined for all $g\in\mathbb{D}[0,t]$ by
    \[  
    \phi_{k,f}(g)=k\ [(d_\text{Sko}(g, B_{\delta_f}(f)))\wedge 1].
    \]
    Then 
    \begin{multline}
        \int_\R K^{\beta^K_0(x)}\E_{\mu^K_{x,T}}\Big[\exp\Big(\log K\int_0^t R(Y_s^K) ds \Big)\mathbbm{1}_{(Y^K_s)_{s\in[0,t]}\in B_{\delta_f}(f)}\Big]dx \\ 
        \begin{aligned}
            & \leq K^{\beta_0(f(0))+\int_0^t R(f(s))ds+\varepsilon}\int_{f(0)-\delta_{f}}^{f(0)+\delta_{f}}\mu^K_{x,t}(B_{\delta_f}(f))dx \\
            & \leq 2\delta_f K^{\beta_0(f(0))+\int_0^t R(f(s))ds+\varepsilon}\sup_{x\in[f(0)-\delta_f,f(0)+\delta_f]}\mathbb{E}_{\mu^K_{x,t}}(K^{-\phi_{k,f}}). 
        \end{aligned}\label{eq:blablabla}
    \end{multline}
    Since the function $\phi_{k,f}$ is bounded continuous, it follows from the Laplace principle uniform on compact sets of Theorem~\ref{thm:main-PGD-bis} that
    \begin{multline*}
        \limsup_{K\to\infty}\frac{1}{\log K}\log\left(\int_\R K^{\beta^K_0(x)}\E_{\mu^K_{x,T}}\Big[\exp\Big(\log K\int_0^t R(Y_s^K)ds\Big)\mathbbm{1}_{(Y^K_s)_{s\in[0,t]}\in B_{\delta_f}(f)} \Big]dx\right) \\   
 \leq \beta_0(f(0))+\int_0^t R(f(s))ds+\varepsilon-\inf_{g\in\mathbb{D}[0,t]}(\phi_{k,f}(g)+I_{t,x}(g)).
    \end{multline*}
    Now, using the definition of $\phi_k$ and~\eqref{eq:I-sci}
    \begin{equation*}
    \inf_{g\in\mathbb{D}[0,t]}(\phi_{k,f}(g)+I_{t,x}(g))\geq \inf_{g\not\in B_{2\delta_f}(f)}\phi_{k,f}(g)\wedge \inf_{g\in B_{2\delta_f}(f)}I_{t,x}(g) \geq ((\delta_f\wedge 1) k)\wedge (I_{t}(f)-\varepsilon).   
        \end{equation*}
        Since $k$ was arbitrary, choosing it large enough entails
    \begin{multline}
        \limsup_{K\to\infty}\frac{1}{\log K}\log\left(\int_\R K^{\beta^K_0(x)}\E_{\mu^K_{x,T}}\Big[\exp\Big(\log K\int_0^t R(Y_s^K)ds\Big)  \mathbbm{1}_{(Y^K_s)_{s\in[0,t]}\in B_{\delta_f}(f)}\Big]dx\right) \\   
 \leq F_t(f)+2\varepsilon. \label{eq:borne-sup-esp-boule}
    \end{multline} 

    Since $A\cap A_0$ is compact, there exists $m<\infty$ and $f_1,\ldots, f_m\in A\cap A_0$ such that
    \[
    A\cap A_0\subset\bigcup_{j=1}^m B_{\delta_{f_j}}(f_j).
    \]
    Now, it follows from Proposition~\ref{prop:FeynmanKac-bis} that
    \begin{align*}
        \E N^{K,A}_t 
        & \leq \sum_{j=1}^m\int_\R K^{\beta^K_0(x)}\E_{\mu^K_{x,T}}\Big[\exp\Big(\log K\int_0^t R(Y_s^K)  ds \Big)\mathbbm{1}_{(Y^K_s)_{s\in[0,t]}\in B_{\delta_{f_j}}(f_j)}\Big]dx \\
        & +\int_{B^c} K^{\beta^K_0(x)+(\bar{b}+\bar{p})t}\mu^K_{x,T}(A_0^c)dx +\text{Leb}(B) K^{\sup_{y\in\R}\beta^K_0(y)+(\bar{b}+\bar{p})t}\sup_{x\in B}\mu^K_{x,T}(A_0^c).
    \end{align*}
    Therefore, using Lemma~\ref{lem:expo-tight} and Eq.~\eqref{as:beta0},~\eqref{eq:borne-B},~\eqref{eq:choix-M} and~\eqref{eq:borne-sup-esp-boule},
    \begin{multline*}
        \limsup_{K\to\infty}\frac{1}{\log K}\log \E N^{K,A}_t \leq \max\{F_t(f_j)+2\varepsilon,\ 1\leq j\leq m\}\vee \left(\sup_{f\in A}F_t(f)-1\right) \\ \vee \left(\sup_{y\in\R,\,K\geq 1}\beta^K_0(y)+(\bar{b}+\bar{p})t-M\right)\leq \sup_{f\in A} F_t(f)+2\varepsilon.
    \end{multline*}
    Since $\varepsilon>0$ was arbitrary, the proof is completed in the case where $\sup_{f\in A}F_t(f)>-\infty.$ \\

    In the case where $\sup_{f\in A}F_t(f)=-\infty$, let $C>0$ be fixed. For all $f\in A$, there exists $\delta_f>0$ 
    such that, for all $x\in[f(0)-\delta_f,f(0)+\delta_f]$ and all $g\in B_{2\delta_f}(f)$, $I_{t,x}(g)\geq M$, where
    \[
    M=\sup_{K\geq 1}\sup_{x\in\R}\beta^K_0(x)+(\bar{b}+\bar{p})t+C.
    \]
    Following the same argument as in~\eqref{eq:blablabla} (with $k=1$), we deduce that
    \[
    \limsup_{K\to\infty}\frac{1}{\log K}\log\left(\int_\R K^{\beta^K_0(x)}\E_{\mu^K_{x,T}}\Big[\exp\Big(\log K\int_0^t R(Y_s^K)ds\Big)\mathbbm{1}_{(Y^K_s)_{s\in[0,t]}\in B_{\delta_f}(f)} \Big]dx\right) \leq -C
    \]
    We conclude as above that 
    \[
    \limsup_{K\to\infty}\frac{1}{\log K}\log \E N^{K,A}_t \leq -C.
    \]
    Since $C>0$ was arbitrary, the proof of the upper bound in Theorem~\ref{thm27bis} is complete.
\end{proof}

\begin{proof}[Proof of the lower bound in Theorem~\ref{thm27bis}]
    We fix $t>0$ and $G\subset\mathbb{D}[0,t]$ open. The proof is divided into 3 steps.
    \medskip

    \noindent\textit{Step 1.} We first prove using the Laplace principle uniform on compacts the following property: for all $f\in\mathbb{D}[0,t]$ such that $I_{t,f(0)}(f)<+\infty$ and all $\delta>0$,
    \begin{equation}
        \label{eq:step1-LB-expectation}
        \liminf_{K\to\infty}\inf_{x\in[f(0)-\delta/3,f(0)+\delta/3]}\frac{1}{\log K}\log\mu^K_{t,x}(B_\delta(f))\geq -\sup_{x\in[f(0)-\delta/3,f(0)+\delta/3]}\inf_{g\in B_{2\delta/3}(f)}I_{t,x}(g).
    \end{equation}

    For all integer $k\geq 1$, we define the function $\phi_k$ on $\mathbb{D}[0,t]$ as: for all $g\in\mathbb{D}[0,t]$
    \[
    \phi_k(g)=k\left(\frac{3}{\delta}d_\text{Sko}(g,B_{2\delta/3}(f))\wedge 1\right).
    \]
    Since $\phi_k$ is bounded continuous, it follows from the Laplace principle uniform on compacts that
    \begin{align*}
            \liminf_{K\to\infty}\inf_{x\in[f(0)-\delta/3,f(0)+\delta/3]}\frac{1}{\log K}\log\E_{\mu^K_{t,x}}(K^{-\phi_k}) & \geq -\sup_{x\in[f(0)-\delta/3,f(0)+\delta/3]} \inf_{g\in\mathbb{D}[0,t]}(\phi_k(g)+I_{t,x}(g)) \\
            & \geq - \sup_{x\in[f(0)-\delta/3,f(0)+\delta/3]} \inf_{g\in B_{2\delta/3}(f)}I_{t,x}(g),
    \end{align*}
    where we used that $\phi_k=0$ on $B_{2\delta/3}(f)$. We observe that, for all $x\in\R$,
    \[
    \E_{\mu^K_{t,x}}(K^{-\phi_k})\leq K^{-k}+\mu^K_{t,x}(B_\delta(f)),
    \]
    so we deduce that
    \begin{multline*}    
    \max\left\{\liminf_{K\to\infty}\inf_{x\in[f(0)-\delta/3,f(0)+\delta/3]}\frac{1}{\log K}\log\mu^K_{t,x}(B_\delta(f))\,;\, -k\right\} \\ \geq -\sup_{x\in[f(0)-\delta/3,f(0)+\delta/3]} \inf_{g\in B_{2\delta/3}(f)}I_{t,x}(g).
    \end{multline*}
    Letting $k\to\infty$ yields~\eqref{eq:step1-LB-expectation}.
    \medskip

    \noindent\textit{Step 2.} We now prove the following weak continuity property of $I_{t,x}$: for all $f\in\mathbb{D}[0,t]$ such that $I_{t,f(0)}(f)<+\infty$ and all $\varepsilon>0$, there exists $\delta>0$ such that, for all $x\in[f(0)-\delta/2,f(0)+\delta/2]$, there exists $g\in B_\delta(f)$ with $g(0)=x$ such that $I_{t,x}(g)\leq I_{t,f(0)}(f)+\varepsilon$. 
    
    To prove this, we recall from Step 1 of the proof of Lemma~\ref{lem:fct-tx}~(ii) that $I_{t,f(0)}(f)<+\infty$ implies that $|\dot{f}|(H')^{-1}(|\dot{f}|/\bar{p})\in L^1[0,t]$. Defining for all $x\in\R$ and $s\in[0,t]$ $f^{(x)}_s=x-f_0+f_s$ and observing that $\dot{f}^{(x)}=\dot{f}$, it then follows from~\eqref{eq:uocuoc} and~\eqref{eq:borne-L-2} that there exists a constant $C$ such that, for all $x$ in a neighborhood of $f(0)$, $L(f^{(x)}_s,\dot{f}^{(x)}_s)\leq C |\dot{f}|(H')^{-1}(|\dot{f}|/\bar{p})$. Since $L(\cdot,\cdot)$ is continuous with respect to both variables, we deduce from Lebesgue's theorem that $I_{t,x}(f^{(x)})\to I_{t,f(0)}(f)$ when $x\to f(0)$, hence the result.

    \medskip

    \noindent\textit{Step 3.} We now conclude the proof as follows: first, if $\inf_{g\in G}I_{t,g(0)}(g)=+\infty$, there is nothing to prove. So assume the converse and fix $\varepsilon>0$. Then $\sup_{g\in G}F_t(g)>-\infty$, so we can take $f\in G$ such that
    \[
    F_t(f)\geq \sup_{g\in G}F_t(g) -\varepsilon.
    \]
    Since $G$ is open and in view of Steps 1 and 2, there exists $\delta>0$ such that $B_\delta(f)\subset G$ and
    \[
    \liminf_{K\to\infty}\inf_{x\in[f(0)-\delta/3,f(0)+\delta/3]}\frac{1}{\log K}\log\mu^K_{t,x}(B_\delta(f))\geq -I_{t,f(0)}(f)-\varepsilon.
    \]
    Reducing $\delta>0$ if necessary, we can assume that, for $K$ large enough, for all $g\in B_\delta(f)$,
    \[
    \beta^K_0(g(0))+\int_0^t R(g(s))ds\geq \beta^K_0(f(0))+\int_0^t R(f(s))ds-\varepsilon.
    \]
    Now, the Feynman-Kac formula of Proposition~\ref{prop:FeynmanKac-bis} implies that
    \begin{align*}
        & \frac{1}{\log K}\log \E N^{K,G}_t \\ & \geq\frac{1}{\log K}\log \int_{f(0)-\delta/3}^{f(0)+\delta/3} K^{\beta^K_0(x)}\E_{\mu^K_{t,x}}\left[\exp\left(\log K\int_0^t R(Y^K_s)ds\right)\mathbbm{1}_{(Y^K_s)_{s\in[0,t]}\in B_\delta(f)}\right] \\ & \geq \beta^K_0(f(0))+\int_0^t R(f(s))ds-\varepsilon+\frac{1}{\log K}\log\left(\frac{2\delta}{3}\inf_{x\in [f(0)-\delta/3,f(0)+\delta/3]}\mu^K_{t,x}(B_\delta(f))\right).
    \end{align*}
    Therefore,
    \[
    \liminf_{K\to\infty}\frac{1}{\log K}\log \E N^{K,G}_t\geq F_t(f)-2\varepsilon\geq \sup_{g\in G}F_t(g)-3\varepsilon.
    \]
 Since $\varepsilon>0$ was arbitrary, the lower bound in Theorem~\ref{thm27bis} is proved.
\end{proof}

\section{Proof of Theorem \ref{thm:bornesup}}
\label{sec:upper-bound}

The proof relies on the following proposition.

\begin{prop}
\label{prop:BS1}
For any $t>0$ and all closed set $A\subset \mathbb{D}[0,t]$, in probability
$$\limsup_{K\to\infty}\frac{\log N^{K,A}_t}{\log K} \leq \sup\{F_t(f) ; f\in A\}.$$
\end{prop}

\begin{proof} We introduce the notation
\[
F_t(A)=\sup\{F_t(f); f\in A\}.
\]
Using  Markov inequality and the upper bound in Theorem \ref{thm27bis}, we have for all fixed $\delta>0$
\begin{align*}
    \limsup_{K\to\infty}\frac{1}{\log K} \log\P(N^{K,A}_t \ge K^{F_t(A)+\delta}) \le  \limsup_{K\to\infty}\frac{1}{\log K} \log\left(\frac{\E(N^{K,A}_t)} {K^{F_t(A)+\delta}}\right)
    \le - \delta.
\end{align*}
Therefore, for $K$ large enough,
$$ \P(N^{K,A}_t \ge K^{F_t(A)+\delta}) \le K^{-\delta/2}.$$
Hence it follows that, in probability,
$$\limsup_{K\to\infty}\frac{\log N^{K,A}_t}{\log K} \leq F_t(A)+ \delta.$$
Since $\delta$ was arbitrary, the proof is complete.
\end{proof}

\begin{proof}[Proof of Theorem~\ref{thm:bornesup}]

\noindent\textit{Step 1: Proof  for $A$ compact.} 

Fix $k\geq 1$. For any $f\in\mathbb{D}[0,t]$ such that $I_t(f)<\infty$, there exists $\delta_f>0$ such that, for all $g\in \text{adh}(B_{\delta_f}(f))$, 
\begin{equation}
\forall s\in [0,t],\quad F_s(g)\leq F_s(f)+\frac{1}{k}.\label{eq:uniform-upp-cont}
\end{equation}
Indeed, if it was not true, there would exist $s_n\in[0,t]$ and $g_n\in\mathbb{D}[0,t]$ converging to $f$ such that $F_{s_n}(g_n)>F_{s_n}(f)+1/k$. After extraction, we can assume that $s_n\to s$. We then deduce from the lower semi-continuity of $I_{(s-\delta)\vee 0}$ (see \cite[Thm.\,3.2.1]{Fathi-book}) that, for all $\delta>0$,
\begin{align*}
    F_{(s-\delta)\vee 0}(f) & \geq \limsup_n F_{(s-\delta)\vee 0}(g_n) \\ & \geq \limsup_n F_{s_n}(g_n)-\bar{R}\delta \\ & \geq \limsup_n F_{s_n}(f)-\bar{R}\delta+\frac{1}{k}=F_s(f)-\bar{R}\delta+\frac{1}{k}.
\end{align*}
Therefore,
\[
F_s(f)=\lim_{\delta\to 0}F_{(s-\delta)\vee 0}(f) \geq F_s(f)+\frac{1}{k},
\]
which is absurd.\\

Similarly, if $f\in \mathbb{D}[0,t]$ is such that $I_t(f)=+\infty$, there exists $\delta_f>0$ such that, for all $g\in \text{adh}(B_{\delta_f}(f))$, 
\[
\forall s\in [0,t],\quad F_s(g)\leq -k.
\]\\

Now, assume that $A$ is compact. Since
\[
A\subset\bigcup_{f\in A}B_{\delta_f}(f),
\]
there exists $n\geq 1$ and $f_1,\ldots,f_n\in A$ such that, denoting $\delta_i=\delta_{f_i}$,
\[
A\subset\bigcup_{i=1}^nB_{\delta_i}(f_i).
\]

Given $i\in\{1,\ldots,n\}$, if there exists $s\leq t$ such that $F_s(f_i)<-1/k$, we can apply Proposition~\ref{prop:BS1} to show that, in probability,
\[
\limsup_{K\to\infty}\frac{\log N^{K,\text{cl}(B_{\delta_i}(f_i))}_s}{\log K}\leq F_s(f_i)+1/k<0.
\]
Since $N^{K,\text{cl}(B_{\delta_i}(f_i))}_s$ is integer-valued, this implies that, $N^{K,\text{cl}(B_{\delta_i}(f_i))}_s=0$ for $K$ large enough, hence with probability converging to one,
\[
\frac{\log N^{K,\text{cl}(B_{\delta_i}(f_i))}_s}{\log K}=-\infty
\]
and, since $t\geq s$,
\[
\frac{\log N^{K,\text{cl}(B_{\delta_i}(f_i))}_t}{\log K}=-\infty.
\]
If $F_t(f_i)>-\infty$ and for all $s\leq t$, $F_s(f_i)\geq-1/k$, we deduce from Proposition~\ref{prop:BS1} that, in probability,
\[
\limsup_{K\to\infty}\frac{\log N^{K,\text{cl}(B_{\delta_i}(f_i))}_t}{\log K}\leq F_t(f_i)+1/k.
\]
Finally, if $F_t(f_i)=-\infty$, we deduce from Proposition~\ref{prop:BS1} that, in probability,
\[
\limsup_{K\to\infty}\frac{\log N^{K,\text{cl}(B_{\delta_i}(f_i))}_t}{\log K}\leq -k.
\]
Observing that
\[
N_t^{K,A}\leq \sum_{i=1}^n N^{K,\text{cl}(B_{\delta_i}(f_i))}_t,
\]
we obtain in probability
\begin{align}
    \limsup_{K\to\infty}\frac{\log N^{K,A}_t}{\log K} & \leq\sup\left\{F_t(f_i)+\frac{1}{k},\ i\in\{1,\ldots,n\}\text{ s.t.\ }F_s(f_i)\geq-\frac{1}{k}\ \forall s\in[0,t]\right\}\vee(-k) \nonumber \\ & \leq\left(\frac{1}{k}+\sup\left\{F_t(f),\ f\in A\text{ s.t.\ }F_s(f)\geq -\frac{1}{k}\ \forall s\in[0,t]\right\}\right)\vee(-k) .\label{eq:preuve-borne-sup-2}
\end{align}
Now, when $k$ converges to $+\infty$,
\[
\sup\left\{F_t(f),\ f\in A\text{ s.t.\ }F_s(f)\geq -\frac{1}{k}\ \forall s\in[0,t]\right\}\xrightarrow[]{}\sup\left\{F_t(f),\ f\in A\text{ s.t.\ }F_s(f)\geq 0\ \forall s\in[0,t]\right\},
\]
since otherwise, there would exist $\eta>0$ and a sequence $(f_k)_{k\geq 1}$ in $A$ such that $F_s(f_k)\geq-1/k$ for all $s\in[0,t]$ and $F_t(f_k)\geq \sup\{F_t(f),\ f\in A,\ F_s(f)\geq 0\ \forall s\in[0,t]\}+\eta$. Since $A$ is compact, we can assume after extraction that $f_k$ converges to some $g\in A$. Since $F_s$ is upper semi-continuous for all $s\in[0,t]$, $F_s(g)\geq\limsup_k F_s(f_k)\geq 0$ for all $s\in[0,t]$ and $F_t(g)\geq \sup\{F_t(f),\ f\in A,\ F_s(f)\geq 0\ \forall s\in[0,t]\}+\eta$, which is a contradiction.

Hence, letting $k\to\infty$ in~\eqref{eq:preuve-borne-sup-2} ends the proof of Theorem~\ref{thm:bornesup} for $A$ compact.

\medskip\noindent\textit{Step 2: Proof of Theorem~\ref{thm:bornesup} for $A$ closed.} Following standard arguments in large deviations, to conclude, it only remains to prove that, for all $M_0$, there exists a compact set $C_{M_0}\subset\mathbb{D}[0,t]$ such that, almost surely
\begin{equation}
\label{eq:pf-borne-sup-3}
    \limsup_{K\to+\infty} \frac{\log N^{K,C_{M_0}^c}_t}{\log K}\leq -M_0.
\end{equation}
This can be deduced from Proposition~\ref{prop:FeynmanKac-bis} and Lemma~\ref{lem:expo-tight} as follows: let $C_{M_0}\subset\mathbb{D}[0,t]$ be the compact set of Lemma~\ref{lem:expo-tight} for $B=\{x\in\R,\ \beta_0(x)\geq -2-t\bar{R}-M_0\}$ (which is compact by Assumption \eqref{as:beta0}), 
and $M=M_0+t\bar{R}+\bar{\beta}+2$, i.e.
\[
\limsup_{K\to\infty}\sup_{x\in B}\frac{1}{\log K}\log\mu^K_{x,t}(C_{M_0}^c)\leq -M_0-\bar{\beta}-t\bar{R}-2.
\]
By Proposition~\ref{prop:FeynmanKac-bis}, for $K$ large enough,
\[
\mathbb{E} N^{K,C_{M_0}^c}_t\leq \int_{B^c} K^{(-2-t\bar{R}-M_0)\wedge(\bar{\beta}-\alpha|x|)}K^{t\bar{R}}dx+\text{Leb}(B)\sup_{x\in B} K^{\bar{\beta} +t\bar{R}+\frac{1}{2}}\mu^K_{x,t}(C_{M_0}^c),
\]
so
\[
\limsup_{K\to\infty}\frac{\log \mathbb{E} N^{K,C_{M_0}^c}_t}{\log K}\leq -M_0-\frac{3}{2}.
\]
To conclude, we proceed as in the proof of Proposition~\ref{prop:BS1}:
\[
\P\left(N^{K,C_{M_0}^c}_t\geq K^{-M_0}\right) \leq\frac{\E N^{K,C_{M_0}^c}_t}{K^{-M_0}}\leq K^{-5/4}
\]
for $K$ large enough. Then Borel-Cantelli's lemma entails~\eqref{eq:pf-borne-sup-3}.
\end{proof}

\section{Proof of Theorem~\ref{thm:borneinf}}
\label{sec:lower-bound}

We follow the general approach of~\cite{berestyckibrunetharrisharrisroberts}, using moment estimates following ideas in \cite{mallein}. We start with considering for $t\geq 0$, $\varepsilon>0$ and $f\in \D[0,t]$ such that $I_t(f)<+\infty$. Recall that $B_\varepsilon(f)$ is the open ball $\{g\in \D[0,t],\ \dSko(g,f)<\varepsilon\}$. 
Let us denote by
$$N^{K,\varepsilon,f}_t=N^{K,B_\varepsilon(f)}_t 
= \sum_{u\in V^K_{t\log K}} \ind_{B_\varepsilon(f)}(X^{K,u}_{s \log K}, s\leq t), $$
the number of particles remaining in the tube of width $\varepsilon$ around $f$ until time $t$.
The key ingredient of the proof is the following lemma.

\begin{lemma}\label{lemma:tube}
Let us consider $t\geq 0 $ and $f\in AC[0,t]$ such that $I_t(f)<+\infty$ and such that for all $s\leq t$, $F_s(f)>0$. Almost surely, for all sufficiently small $\varepsilon>0$, 
    \begin{equation}
        \liminf_{K\rightarrow +\infty} \frac{1}{\log K} \log N^{K,\varepsilon,f} \geq F_t(f).\label{etape17}
    \end{equation}
\end{lemma}
    
Assuming Lemma \ref{lemma:tube}, we can then complete the proof of Theorem~\ref{thm:borneinf} as follows.
Let $t>0$ and $G$ be an open subset of $\mathbb{D}[0,t]$. If
\[
\sup\{F_t(g); g\in G, \; \forall s\in [0,t],\; F_s(g)> 0\}=-\infty,
\]
then we have nothing to prove. Otherwise, for any $\delta>0$, we can find $f\in G$ such that $F_s(f)>0$ for all $s\in[0,t]$ and
\[
F_t(f)\geq \sup\{F_t(g); g\in G, \; \forall s\in [0,t],\; F_s(g)> 0\}-\delta.
\]
Since $G$ is open, there exists $\varepsilon>0$ such that $B_\varepsilon(f)\subset G$. Hence, reducing $\varepsilon>0$ if necessary, we deduce from Lemma~\ref{lemma:tube} that, almost surely,
\begin{align*}
    \liminf_{K\to\infty}\frac{1}{\log K}\log N^{K,G}_{t} & \geq \liminf_{K\to\infty}\frac{1}{\log K}\log N^{K,\varepsilon,f}_{t} \\ & \geq F_t(f) \\ & \geq \sup\{F_t(g); g\in G, \; \forall s\in [0,t],\; F_s(g)> 0\}-\delta.
\end{align*}
Since $\delta>0$ was arbitrary, the proof of Theorem~\ref{thm:borneinf} is completed. \\

The proof of Lemma \ref{lemma:tube} is divided into two subsections. In Section \ref{sec:estimateNfeps}, we establish an upper bound of $\E\big[\big(N^{K,\varepsilon,f}_t\big)^2\big]$ with the idea of \cite{mallein} in mind, by establishing a many-to-two formula (or formula for forks). We conclude in Section \ref{sec:preuve_lemme_5.1}.

\subsection{Estimates for the second moment of $N^{K,\varepsilon,f}$}\label{sec:estimateNfeps}

A key ingredient to compare $N^{K,\varepsilon,f}_t$ with its expectation is to control its second moment. For this purpose, we first establish the next lemma.

\begin{lemma} \label{lem-appli-fourches}
Let us consider $t\geq 0 $ and $f\in AC[0,t]$ such that $I_t(f)<+\infty$ and such that for all $s\leq t$, $F_s(f)>0$. For all $\delta'>0$, for all sufficiently small $\varepsilon>0$ and sufficiently large $K$, we have
    \begin{equation}
\E\big[\big(N^{K,\varepsilon,f}_t\big)^2\big]\leq \E^2\big[N^{K,\varepsilon,f}_t\big] \times K^{8\delta'}.\label{but:mallein}
     \end{equation}
\end{lemma}
\begin{proof}Recall that $V^K_t$ is the collection of labels of individuals alive at time $t$, and that $V^K_{[0,t]}$ consists of the labels of individuals alive between times $0$ and $t$.\\

    First, we have $\E\big[\big(N^{K,\varepsilon,f}_t\big)^2\big] =  A+B+C$ with
    \begin{align*}
        A= &  \E\Big(\sum_{v\in V^K_{t\log K}} \ind^2_{B_\varepsilon(f)}(X^{K,v}_{s\log K},s\leq t)\Big)\\
        B= & \E\Bigg(\sum_{{\scriptsize\begin{array}{c}v_1\not=v_2 ,\ v_1,v_2\in {V}^K_{t \log K}\\
        \exists w \in V^K_{[0,t\log  K ]}, w\prec v_1, w\prec v_2\end{array}}} \ind_{B_\varepsilon(f)}(X^{K,v_1}_{s\log K},s\leq t)\ind_{B_\varepsilon(f)}(X^{K,v_2}_{s\log K},s\leq t)\Bigg)\\
        C= & \E\Bigg(\sum_{{\scriptsize\begin{array}{c}v_1\not=v_2, \ v_1,v_2\in {V}^K_{t \log K}\\
        \not\exists w \in V^K_{[0,t\log  K ]}, w\prec v_1, w\prec v_2 \end{array}}} \ind_{B_\varepsilon(f)}(X^{K,v_1}_{s\log K},s\leq t)\ind_{B_\varepsilon(f)}(X^{K,v_2}_{s\log K},s\leq t)\Bigg).
    \end{align*}In $B$, the two individuals $v_1$ and $v_2$ have a common ancestor $w\in V^K_{[0,t\log  K ]}$ and hence a common ancestor at time $0$, whereas in $C$, the two individuals are descended from different ancestors at time $0$.\\

\noindent \textbf{Term $A$.}  Let $\delta'>0$. For the term $A$, we notice that 
                $A=  \E\big(N_t^{K,\varepsilon,f}\big)$. 
                Then, using Theorem \ref{thm27bis}, we obtain that for $\varepsilon$ small enough and $K$ large enough, 
                \begin{equation}
                   K^{F_t(f)-\delta'}\leq K^{\sup_{g\in B_\varepsilon(f)} F_t(g)-\delta'} \leq A \leq K^{\sup_{g\in \overline{B}_\varepsilon(f)} F_t(g)+\delta'}\leq K^{F_t(f)+2\delta'},\label{etape8}
                \end{equation}
                where $\overline{B}_\varepsilon(f)$ denotes here the closure in $\D$ of $B_\varepsilon(f)$, and where we used \eqref{eq:I-sci} for the right-most inequality.\\
            
\noindent \textbf{Term B.} 
Summing over all pair of leaves $(v_1,v_2)\in (V^K_{t\log K})^2$, $v_1\not=v_2$, that have a common ancestor at time $0$ amounts to summing over all the internal nodes $w$ of the tree, and then over all pairs made of one descendant at $t\log K$ of $w0$ and one of $w1$. Then, denoting by $\widetilde{S}_w \log K$ the time at which the individual $w$ gives birth, 
\begin{align*}
B = & \E \left[\sum_{\stackrel{w \in V^K_{[0,t\log K]}}{\widetilde{S}_w< t}}  \left(\sum_{\stackrel{v_1\in V^K_{t\log K}}{v_1 \succeq w0}} \ind_{B_\varepsilon(f)}\big( X^{K,v_1}_{s\log K}, s\leq t\big)\right)\times\left( \sum_{\stackrel{v_2\in V^K_{t\log K}}{v_2 \succeq w1}} \ind_{B_\varepsilon(f)}\big(X^{K,v_2}_{s \log K}, s\leq t\big)  \right)\right].
 \end{align*}
 
 Conditionally on $\mathcal{F}_{\widetilde{S_w}\log K}$, the two populations descending from $w0$ and $w1$ are independent. The idea is to condition on $\mathcal{F}_{\widetilde{S}_w\log K}$ and consider the time intervals $ [0,\widetilde{S}_w\log K]$ and $[\widetilde{S}_w\log K,t\log  K ]$.
It is then natural to introduce, for any time $s\in [0,t]$ and $f\in AC[0,t]$, the notation 
 \begin{equation}F_{s,t}(f):=F_t(f)-F_s(f)= \int_s^t R(f(u))du - I_{s,t}(f),\label{decomposition_F}\end{equation}
where $I_{s,t}(f)=I_t(f)-I_s(f)=\int_s^t L(f_u,\dot{f}_u) du$ can be interpreted as the rate function of the large deviations for $Y^K$ on the time interval $[s,t]$. 

A difficulty arises from the Skorohod distance as,
\begin{align}\dSko(f,g ) \leq \max \Big(\dSko\big(f|_{[0,s]},g|_{[0,s]} \big), \dSko\big(f|_{[s,t]},g|_{[s,t]} \big) \Big),\label{etape20}\end{align}with an inequality that can be strict as the right hand side can be associated only with a time deformation $\lambda$ in \eqref{def:dSko} that keeps $s$ fixed.
This implies that 
\[
B_\varepsilon(f) \supset \left\{g\in\mathbb{D}[0,t]: g|_{[0,s]}\in B_\varepsilon(f|_{[0,s]})\right\} \cap \left\{g\in\mathbb{D}[0,t]: g|_{[s,t]}\in B_\varepsilon(f|_{[s,t]})\right\},
\]
and not the reverse inclusion, which raises difficulties when separating the time interval $[0,t]$ into $[0,s]$ and $[s,t]$. \\
However, we will be only interested in the case where $f \in AC[0,t]$, for which there exists a constant  $\eta(\varepsilon)$ depending on $\varepsilon$, $f$ and $t$ (the dependencies on $f$ and $t$ are omitted in the notation for the sake of readability), such that
\begin{equation}\label{etape21}
    B_\varepsilon(f) \subset B_{\eta(\varepsilon)}^{\infty}(f):= \{g\in \D[0,t],\ \sup_{r\in [0,t]} |f(r)-g(r)|< \eta(\varepsilon)\}.
\end{equation}This inclusion is proved in Appendix \ref{app:dsko-unif}. Now, for the uniform norm, we have an equality in \eqref{etape20} and
\begin{equation}\label{etape22}B_{\eta(\varepsilon)}^\infty(f)= 
\left\{g\in\mathbb{D}[0,t]: g|_{[0,s]}\in B^\infty_{\eta(\varepsilon)}(f|_{[0,s]})\right\} \cap \left\{g\in\mathbb{D}[0,t]: g|_{[s,t]}\in B^\infty_{\eta(\varepsilon)}(f|_{[s,t]})\right\}.
      \end{equation}

With these notations, we have on the set $\{\widetilde{S}_w<t\}$:
 
 \begin{align}
\lefteqn{  \E\Big(\sum_{\stackrel{v\in V^K_{t\log K}}{v \succeq w0}} \ind_{B_\varepsilon(f)}\big( X^{K,v}_{s\log K}, s\leq t\big)  \, \Big|\, \mathcal{F}_{\widetilde{S}_w\log K}\Big)}\nonumber\\
     \leq & \E\Big(\sum_{\stackrel{v\in V^K_{t\log K}}{v \succeq w0}} \ind_{B^\infty_{\eta(\varepsilon)}(f)}\big( X^{K,v}_{s\log K}, s\leq t\big)  \, \Big|\, \mathcal{F}_{\widetilde{S}_w\log K}\Big) \nonumber\\
     = & \ind_{B^\infty_{\eta(\varepsilon)}(f |_{[0,\widetilde{S}_w]})}(X^{K,w}_{r\log K}, r\leq \widetilde{S}_w) \nonumber\\
     & \hspace{2cm} \times \E_{\delta_{X^{K,w0}_{\widetilde{S}_w\log K}}}\Big(\sum_{v\in V^K_{(t-s)\log K}} \ind_{B^\infty_{\eta(\varepsilon)}(f(s+\cdot))}\big( X^{K,v}_{r\log K}, r\leq t-s\big)\Big)|_{s={\widetilde{S}_w}}\nonumber\\
     = & 
     \ind_{B^\infty_{\eta(\varepsilon)}(f |_{[0,\widetilde{S}_w]})}(X^{K,w}_{r\log K}, r\leq \widetilde{S}_w) \nonumber\\
     & \hspace{2cm} \times \E_{X^{K,w0}_{\widetilde{S}_w\log K}}\Big(K^{\int_{0}^{t-s} R(X^{K}_{r\log K})dr} \ind_{B^\infty_{\eta(\varepsilon)}(f(s+\cdot))}\big(X^{K}_{r\log K}, r\leq t-s)  \Big)|_{s={\widetilde{S}_w}} \nonumber\\
         \leq &          \ind_{B^\infty_{\eta(\varepsilon)}(f |_{[0,\widetilde{S}_w ]})}(X^{K,w}_{r\log K}, r\leq \widetilde{S}_w) \, K^{\int_{\widetilde{S}_w}^t R(f(r)) \ dr + C(f,R)t\eta(\varepsilon)} \nonumber\\
         & \hspace{2cm} \times \P_{X^{K,w0}_{\widetilde{S}_w\log K}} \Big(\sup_{r\leq t-s} |X^K_{r\log K}-f(s+r)|<\eta(\epsilon)\Big)|_{s=\widetilde{S}_w}, \label{etape25}
     \end{align}
  where we have used Proposition \ref{prop:FeynmanKac-bis}(ii) to obtain the fourth line, and where the constant $C(f,R)$ is the Lipschitz norm of $R$ on $[\inf f-1,\sup f+1]$.

We would like to use Theorem \ref{thm:main-PGD-bis} to upper bound the probability appearing in the right hand side of \eqref{etape25}. But because $\widetilde{S}_w$ is a random time, we have to discretize the time interval $[0,t]$. It will be useful to consider $h>0$ a small time mesh such that $t/h\in\mathbb{N}$. For any $x\in\R$ and $s<s'<t$,
\begin{align*}
&\P_{x} \Big(\sup_{r\leq t-s} |X^K_{r\log K}-f(s+r)|<\eta(\epsilon)\Big)\\
  & \leq  \E_x\left(\ind_{X^K_{(s'-s)\log K}\in[f(s')-\eta(\varepsilon),f(s')+\eta(\varepsilon)]}\P_{X^{K}_{(s'-s)\log K}} \Big(\sup_{r\leq t-s'} |X^K_{r\log K}-f(s'+r)|<\eta(\epsilon)\Big)\right) \\
  &\leq  \sup_{y\in[f(s')-\eta(\varepsilon),f(s')+\eta(\varepsilon)]}\P_{y} \Big(\sup_{r\leq t-s'} |X^K_{r\log K}-f(s'+r)|<\eta(\epsilon)\Big) \\ 
  & \leq\sup_{y\in[f(s')-\eta(\varepsilon),f(s')+\eta(\varepsilon)]}\mu^{K}_{y,t-s'}(B_{\eta(\varepsilon)}(f|_{[s',t]}))
     \end{align*}where we notice that we have returned to a ball with respect to the Skorohod distance in the last line, permitting to use our large deviation result.
Hence, for $k\in\{0,\ldots,t/h-1\}$, on the event $\{\widetilde{S}_w\in[kh,(k+1)h]\}$,
\begin{multline*}
 \E\Big(\sum_{\stackrel{v\in V^K_{t\log K}}{v \succeq w0}} \ind_{B_\varepsilon(f)}\big( X^{K,v}_{s\log K}, s\leq t\big)  \, \Big|\, \mathcal{F}_{\widetilde{S}_w\log K}\Big)\\
         \leq           \ind_{B^\infty_{\eta(\varepsilon)}(f |_{[0,\widetilde{S}_w ]})}(X^{K,w}_{r\log K}, r\leq \widetilde{S}_w) \, K^{\int_{\widetilde{S}_w}^t R(f(r)) \ dr + C(f,R)t\eta(\varepsilon)} \\
          \times \sup_{y\in[f((k+1)h)-\eta(\varepsilon),f((k+1)h)+\eta(\varepsilon)]}\mu^{K}_{y,t-(k+1)h}(B_{\eta(\varepsilon)}(f|_{[(k+1)h,t]})).
  \end{multline*}
Using the Laplace principle uniformly on compact sets of Theorem~\ref{thm:main-PGD-bis} as in~\eqref{eq:def-phi_k} and~\eqref{eq:LDP-unif}, we deduce that, for $K$ large enough, for all $y\in [f((k+1)h)-\eta(\varepsilon),f((k+1)h)+\eta(\varepsilon)]$,
\begin{align*}
 \frac{\log\mu^{K}_{y,t-(k+1)h}(B_{\eta(\varepsilon)}(f|_{[(k+1)h,t]}))}{\log K} 
& \leq -\inf_{\stackrel{g\in \overline{B}_{\eta(\varepsilon)}(f|_{[(k+1)h,t]})}{g(0)\in[f((k+1)h)-\eta(\varepsilon),f((k+1)h)+\eta(\varepsilon)]}} I_{t-(k+1)h}(g) \\
& \leq - I_{t-(k+1)h}(f|_{[(k+1)h,t]})+\delta',
\end{align*}
by the lower semi-continuity of $I_{t-(k+1)h}$, for $\varepsilon$ small enough. Hence, on the event $\{\widetilde{S}_w\in[kh,(k+1)h]\}$,
\begin{align}
 \lefteqn{\E\Big(\sum_{\stackrel{v\in V^K_{t\log K}}{v \succeq w0}} \ind_{B_\varepsilon(f)}\big( X^{K,v}_{s\log K}, s\leq t\big)  \, \Big|\, \mathcal{F}_{\widetilde{S}_w\log K}\Big) }\nonumber\\
         \leq &          \ind_{B^\infty_{\eta(\varepsilon)}(f |_{[0,\widetilde{S}_w ]})}(X^{K,w}_{r\log K}, r\leq \widetilde{S}_w) \, K^{\int_{\widetilde{S}_w}^t R(f(r)) \ dr + C(f,R)t\eta(\varepsilon)-I_{t-(k+1)h}(f|_{[(k+1)h,t]})+\delta'} \nonumber \\ \leq & \ind_{B^\infty_{\eta(\varepsilon)}(f |_{[0,\widetilde{S}_w ]})}(X^{K,w}_{r\log K}, r\leq \widetilde{S}_w) \, K^{F_{\widetilde{S}_w,t}(f)+C(f,R)t\eta(\varepsilon)+\delta'+I_{kh,(k+1)h}(f)} \nonumber \\ 
         \leq & \ind_{B^\infty_{\eta(\varepsilon)}(f |_{[0,\widetilde{S}_w ]})}(X^{K,w}_{r\log K}, r\leq \widetilde{S}_w) \, K^{F_{\widetilde{S}_w,t}(f)+2\delta'},\label{etape26}
     \end{align}
  reducing $\varepsilon>0$ if needed and choosing $h$ such that $I_{kh,(k+1)h}(f)\leq \delta'/2$ for all $k$.
This property holds true on the event $\{\widetilde{S}_w\in[kh,(k+1)h]\}$ for all $K$ large enough and all $\varepsilon>0$ small enough. Since there are only finitely many $k\in\{0,\ldots,t/h-1\}$, we deduce that \eqref{etape26} holds true almost surely for all $K$ large enough and all $\varepsilon>0$ small enough.

  A similar expression holds for the population stemming from $w1$. Thus, using Proposition~\ref{prop:forks}, 
  \begin{align}
B\leq  & \int_\R K^{\beta_0(x)} \mathbb{E}_{\delta_x} \left( \sum_{\stackrel{w \in V^K_{[0,t\log K]}}{\widetilde{S}_w< t}} \ind_{\sup_{r\leq \widetilde{S}_w} |X^{K,w}_{r\log K} -f(r)|<\varepsilon} K^{2 F_{\widetilde{S}_w,t}(f)+4\delta'}\right) \ dx \nonumber\\
    \leq & \int_{f(0)-\varepsilon}^{f(0)+\varepsilon} K^{\beta_0(x)} \int_0^{t} \mathbb{E}_{x} \Big( \ind_{\sup_{r\leq s} |X^K_{r\log K}-f(r)|<\varepsilon}  \nonumber\\
    & \qquad \qquad \qquad \qquad  \times (b+p+d)(X^K_{s\log K}) \ K^{2F_{s,t}(f)+\int_0^s R(X^K_{r\log K})dr + 4\delta' }\Big)\, ds \, dx \nonumber\\
\leq & (\bar{b}+\bar{p}+\bar{d}) 2\varepsilon K^{\beta_0(f(0))+C\varepsilon} \int_0^t K^{-I_s(f)+2 F_{s,t}(f)+\int_0^s R(f(r)) dr + 5\delta'} ds\nonumber \\
\leq &  K^{F_t(f) + \sup_{s\in [0,t]} F_{s,t}(f)+6\delta'},
\end{align}
for $\varepsilon$ sufficiently small and $K$ sufficiently large, and where we used \eqref{def:Ft-fonctionnelle} and 
\eqref{decomposition_F} in the fourth line. 
Note that:
\begin{align*}
    F_t(f)+\sup_{s\in [0,t]}F_{s,t}(f)= & F_t(f)+\sup_{s\in [0,t]} \big( F_t(f)-F_s(f)\big) = 2F_t(f)- \inf_{s\in [0,t]}F_s(f).
\end{align*}
Thus,
 \begin{equation}
B\leq   K^{2F_t(f)-\inf_{s\in [0,t]} F_{s}(f) + 6\delta'}\leq K^{2F_t(f) + 6\delta'},\label{etapeB}
\end{equation}
as $\inf_{s\in [0,t]} F_s(f)\geq 0$ by the assumptions of Lemma \ref{lem-appli-fourches}.\\

\color{black}

\noindent \textbf{Term $C$.} By the branching property,
\begin{align*}
    C = & \E\Bigg(\sum_{\stackrel{w_1\not= w_2}{w_1,w_2\in V_0^K}} \Bigg(\sum_{\stackrel{v_1\in {V}^K_{t \log K}}{v_1\succeq w_1}} \sum_{\stackrel{v_2\in {V}^K_{t \log K}}{v_2\succeq w_2}} \ind_{B_\varepsilon(f)}(X^{K,v_1}_{s\log K},s\leq t)\ind_{B_\varepsilon(f)}(X^{K,v_2}_{s\log K},s\leq t)\Bigg)\Bigg)\\
    = & \E\Bigg(\sum_{\stackrel{w_1\not= w_2}{w_1,w_2\in V_0^K}} \E_{\delta_{X^{K,w_1}_0}}\Big(\sum_{v_1\in {V}^K_{t \log K}} \ind_{B_\varepsilon(f)}(X^{K,v_1}_{s\log K},s\leq t)\Big)\\
    &  \hspace{5cm} \times \E_{\delta_{X^{K,w_2}_0}}\Big(\sum_{v_2\in {V}^K_{t \log K}} \ind_{B_\varepsilon(f)}(X^{K,v_2}_{s\log K},s\leq t)\Big)\Bigg) \\
    = & \E\Bigg(\sum_{\stackrel{w_1\not= w_2}{w_1,w_2\in V_0^K}}\E_{\delta_{X^{K,w_1}_0}}\Big(N^{K,\varepsilon,f}_t\Big)\E_{\delta_{X^{K,w_2}_0}}\Big(N^{K,\varepsilon,f}_t\Big)\Bigg)
    \end{align*}
    Using \eqref{etape8} for the internal expectations and accounting that we start from a single particle and not from a point process of intensity $K^{\beta_0(x)}dx$,
    \begin{align}
    C \leq & \E\Bigg(\iint_{(f(0)-\varepsilon,f(0)+\varepsilon)^2} \mathbbm{1}_{x\neq y}K^{2F_t(f) - 2\beta_0(f(0)) +4\delta'}Z_0^K(dx)Z_0^K(dy)\Bigg) \notag \\ 
    = & K^{2F_t(f) - 2\beta_0(f(0)) +4\delta'}\iint_{(f(0)-\varepsilon,f(0)+\varepsilon)^2} K^{\beta_0(x)} K^{\beta_0(y)} dx\ dy 
    \leq   K^{2F_t(f)+5\delta'}, \label{etapeC}
\end{align} 
where we used the multivariate Mecke's formula (e.g.~\cite[Theorem 4.4]{lastpenrose}) in the second line.

Gathering the upper bounds in \eqref{etape8}, \eqref{etapeB} and \eqref{etapeC} gives that:
\begin{equation}
    \E\big[\big(N^{K,\varepsilon,f}_t\big)^2\big] \leq K^{2F_t(f) + 6\delta'} \leq \E^2\big[N^{K,\varepsilon,f}_t\big] \times K^{8\delta'}, \label{etape13}
\end{equation}where the last inequality comes from the lower bound in \eqref{etape8}. This ends the proof of Lemma \ref{lem-appli-fourches}.
\end{proof}

Notice that if the initial condition has not the intensity $K^{\beta_0(x)}\ dx$ but $K^{\beta_0(x)-\delta''}\ dx$ for example, the proof above would provide instead of \eqref{etape13} that: 
\begin{equation}
\E\big[\big(N^{K,\varepsilon,f}_t\big)^2\big] \leq K^{2F_t(f) +6\delta' -2\delta''} \leq \E^2\big[N^{K,\varepsilon,f}_t\big] \times K^{8\delta' }.        \label{etape12}
\end{equation}
For the last inequality, notice that $\beta_0(x)$ appears in the definition of $F_t(f)$ so that for this new initial condition $\beta_0(x)-\delta''$, the lower bound in \eqref{etape8} becomes:
\begin{equation}
\label{etape78}
    K^{(\beta_0(x)-\delta'')+\int_0^t R(f(s))ds-I_t(f) -\delta' }=K^{F_t(f)-\delta''-\delta'} \leq \E(N_t^{K,\varepsilon,f}).
\end{equation}
These inequalities will be useful in the following.

\subsection{Proof of Lemma \ref{lemma:tube}}\label{sec:preuve_lemme_5.1}

We are now ready to prove Lemma \ref{lemma:tube}. 
Let us introduce
 \begin{equation}\label{def:delta0}
    \delta_0 = \min_{s\in[0,t]} \left\{F_s(f)\right\}>0.
    \end{equation}
Let us fix $\delta\in (0,\delta_0)$. Our purpose is to prove that almost surely, for all sufficiently small $\varepsilon > 0$,
  \begin{equation}\label{but:minoration}
         \liminf_{K\to\infty}\frac{1}{\log K}\log N^{K,\varepsilon,f}_{t} > F_t(f)-\delta.
    \end{equation}
A direct use of Lemma~\ref{lem-appli-fourches} would only show that $\frac{1}{\log K}\log N^{K,\varepsilon,f}_{t} > F_t(f)-\delta$ holds with a probability converging to 0 as $K\to+\infty$ (see~\eqref{etape10} below). To obtain the almost sure lower bound we make use of the branching property by dividing the initial population into several groups. For this, let us consider $\delta'' \in (0,\delta_0)$ so that $\beta_0(f(0))-\delta''>0$. This $\delta''$ as well as the $\delta'$ appearing in \eqref{etape8} will be fixed depending on $\delta$ in the end of the proof, and $\varepsilon$ (resp. $K$) are chosen small (resp. large) enough according to these choices.

By the form of the initial condition \eqref{conditioninitiale}, and by the superposition principle, we can write  
     \begin{equation*}Z^K_0(dx)=\sum_{i=1}^{\lfloor K^{\delta''}\rfloor} \widetilde{Z}^{K,i}_0(dx)\end{equation*}
     where $\widetilde{Z}^{K,i}_0(dx)$ are i.i.d. Poisson point measures with the modified intensity measure $(K^{\beta^K_0(x)}/\lfloor K^{\delta''}\rfloor) dx$. This decomposition combined with the branching property will provide that \eqref{etape17} holds almost surely by Borel-Cantelli's lemma. For $i\in \{1,\cdots ,\lfloor K^{\delta''}\rfloor\}$, we will denote by $\widetilde{N}^{K,\varepsilon,f,i}_t$ the number of particles, among those started from $\widetilde{Z}^{K,i}_0$, that remain in the tube of width $\varepsilon$ around $f$ until time $t$. Note that the random variables $\widetilde{N}^{K,\varepsilon,f,i}_t$ are i.i.d.\\

 Let us consider the branching process started from $\widetilde{Z}^{K,1}_0$ and let $\delta\in (0,\delta_0)$.  We have by the Cauchy-Schwarz inequality that
     \begin{align}\label{etape15}
     \E\Big[\widetilde{N}^{K,\varepsilon,f,1}_t \ind_{\widetilde{N}^{K,\varepsilon,f,1}_t>K^{F_t(f)-\delta}}\Big]\leq & \sqrt{\E\big[\big(\widetilde{N}^{K,\varepsilon,f,1}_t\big)^2\big] \P\big(\widetilde{N}^{K,\varepsilon,f,1}_t > K^{F_t(f)-\delta}\big)}.
     \end{align}
     Moreover,
     \begin{align}
         \E\Big[\widetilde{N}^{K,\varepsilon,f,1}_t \ind_{\widetilde{N}^{K,\varepsilon,f,1}_t>K^{F_t(f)-\delta}}\Big] = & \E\Big[\widetilde{N}^{K,\varepsilon,f,1}_t \Big]-\E\Big[\widetilde{N}^{K,\varepsilon,f,1}_t \ind_{\widetilde{N}^{K,\varepsilon,f,1}_t\leq K^{F_t(f)-\delta}}\Big]\nonumber\\
         \geq & \E\big[\widetilde{N}^{K,\varepsilon,f,1}_t \big]  - K^{F_t(f)-\delta} \nonumber\\
         \geq  & \E\big[\widetilde{N}^{K,\varepsilon,f,1}_t \big] \big( 1  - K^{\delta'+\delta''-\delta}\big),
     \end{align}
     using~\eqref{etape78}. 
     As a consequence, using \eqref{etape12},
     \begin{align*}
     \P\big(\widetilde{N}^{K,\varepsilon,f,1}_t>K^{F_t(f)-\delta}\big)  \geq & \frac{\E^2\big[\widetilde{N}^{K,\varepsilon,f,1}_t \big] \big( 1  - K^{\delta'+\delta''-\delta}\big)^2}{\E\big[\big(\widetilde{N}^{K,\varepsilon,f,1}_t\big)^2\big]}\nonumber\\
     \geq & K^{-8\delta'} \big( 1  - K^{\delta'+\delta''-\delta}\big)^2.
     \end{align*}If $\delta'$ and $\delta''$ are small enough so that
    $\delta'+\delta''-\delta <0$, we obtain for $K$ large enough, 
\begin{equation}
    \P\big(\widetilde{N}^{K,\varepsilon,f,1}_t>K^{F_t(f)-\delta}\big) \geq \frac{1}{2} K^{-8 \delta'}\geq K^{-9\delta'}\label{etape10}
\end{equation} and the same inequality also holds for any $\widetilde{N}^{K,\varepsilon,f,i}_t$, $i\in \{1,\dots, \lfloor K^{\delta''}\rfloor \}$.\\

      We can now conclude, using the branching property. Let us consider the full branching process started from $Z_0^K(dx)$. We have: \begin{align}
\P\big(N^{K,\varepsilon,f}_t\leq  K^{F_t(f)-\delta}\big) 
         \leq &  \prod_{i=1}^{K^{\delta''}}\P\big(\widetilde{N}^{K,\varepsilon,f,i}_t\leq  K^{F_t(f)-\delta}\big) \nonumber\\
         \leq & \Big(1-K^{-9\delta'}\Big)^{\lfloor K^{\delta''}\rfloor }\  
         \sim  \exp\big(- K^{\delta'' - 9\delta'}\big),\label{etape18}
\end{align}for $K$ sufficiently large. This bound converges to zero provided $\delta''-9\delta'>0$.\\

Choosing $\delta'=\delta/20$ and $\delta''=\delta/2$, we have that $\delta'+\delta''-\delta=-9 \delta/20<0$ and $\delta''-9\delta' = \delta/20>0$, in accordance to what we wanted in \eqref{etape10} and \eqref{etape18}.\\
     
Moreover, from \eqref{etape18},  we obtain by Borel-Cantelli's lemma, that \eqref{but:minoration} holds true almost surely.
Since $\delta$ was arbitrary, we have proved Lemma~\ref{lemma:tube}. \hfill $\Box$

\section{The link between the variational formulation of the limit and the Hamilton-Jacobi equation \eqref{eq:HJ}}\label{section:link-variational-HJ}

In this section, we study the link between the variational formulation \eqref{def:u-a} and the Hamilton-Jacobi equation \eqref{eq:HJ}. In particular, we prove Theorem \ref{prop:var-to-visc}, Lemma \ref{lem:ua-cont}, and Theorem \ref{cor:HJ}. To this end, we first provide some preliminary lemmas. 

  \begin{lemma}
  \label{lem:optimal-traj}
 Let $(t,x)\in \widetilde\Omega_a$. There exists an optimal trajectory $f\in \mathrm{AC}([0,t])$ {in the maximizing problem in \eqref{def:u-a}},  such that $f(t)=x$ and that for all $s\in [0,t]$, $F_s(f)\geq a$. Moreover, $\|{f}\|_{W^{1,\infty}([0,t])}$ is uniformly bounded for all $(t,x)\in [0,T]\times B_{ M}(0)$.   
 \end{lemma} 
 \begin{proof}
 Let $f_n\in \mathrm{AC}[0,t]$ be such that $F_t(f_n)\to u_a(t,x)$ and $F_s(f_n)\geq a$, for all $s\in [0,t]$. Since $L(x,v)$ is strictly convex and superlinear with respect to $v$, and since $R(x)$ and $\beta_0(x)$ are bounded above, it can be shown (see~\cite[Section 3.2]{Fathi-book}) that $f_n$ converges, as $n\to +\infty$, to an absolutely continuous trajectory $f_0$, with $f_0(t)=x$ and
 \[
F_s(f_n)\to F_s(f_0) , \qquad \text{for all $s\in [0,t]$}.
 \]
 Since $F_s(f_n)\geq a$, for all $s\in [0,t]$, we also have $F_s(f_0)\geq a$ and hence $(s,f_0(s))\in \widetilde\Omega_a$.
 Moreover, one can prove (see \cite{Fathi-book} and \cite[Lemma 6]{VC.KL:19}) that this optimal trajectory is bounded in $W^{1,\infty}([0,t])$ for all $(t,x)$ in a compact set $[0,T]\times B_{{M}}(0)$. Note that here we can use the results of \cite{VC.KL:19} since our assumptions in \eqref{borne-bp}--\eqref{def:R} 
 lead to the assumptions made in the latter article (see \cite[Corollary 4]{VC.KL:19}). \end{proof}

 \begin{lemma}
 \label{lem:ucont}
 The set $\Omega_a$ is an open set.  Furthermore, $u_a$ is bounded above locally in $t$ and globally in $x$, and it is locally Lipschitz continuous in $\overline \Omega_a$, and consequently in $\widetilde \Omega_a$. Moreover, the Lipschitz bound in $t$ and $x$, is locally uniform with respect to $a$.
 \end{lemma}
 \begin{proof}
 i) {\bf{$u_a$ is lower semicontinuous in $\Omega_a$ and $\Omega_a$ is an open set.}} {    Let $(t_1,x_1)\in \Omega_a$. We prove that  for all $\eta>0$,  there exists $r>0$, small enough, such that $B_r(t_1,x_1)\in \Omega_a$ and that, for all $(t_2,x_2)\in B_r(t_1,x_1)$, we have 
 \[
 u_a(t_2,x_2)\geq u_a(t_1,x_1)-\eta.
 \]}
 Let $(t_2,x_2)\in B_r(t_1,x_1)$. We define $\delta=|t_1-t_2|+|x_1-x_2|$ and note that $t_2-\delta<t_1$. We also recall that, due to Lemma \ref{lem:optimal-traj}, there exists an optimal trajectory $f_1(\cdot)$ such that $f_1(t_1)=x_1$, $u(t_1,x_1)=F_{t_1}(f_1)$ and $F_s(f)\geq a$, for all $s\in [0,t_1]$. We define
  \begin{equation}
  \label{eq:gammaO}
 \overline f (s)=
 \begin{cases}
 f_1(s) & \text{for all $s\in [0,t_2-\delta]$,} \\
 f_1(t_2-\delta)+\frac{s-t_2+\delta} {\delta}( x_2-f_1(t_2-\delta) )&  \text{for all $s\in [ t_2-\delta,t_2]$.}
 \end{cases}
 \end{equation}
  We will prove that, for $r$ small enough, $F_\tau( \overline f )\geq a$ for all $\tau \in [0,t_2]$ and that $F_{t_2}( \overline f ) \geq{u_a(t_1,x_1)-\eta}$. 
Note that, for all $\tau\in [0,t_2-\delta]$, $F_\tau( \overline f )=F_\tau(f_1)\geq a$. We next consider the case $\tau\in [t_2-\delta,t_1]$ and write
\begin{equation}
\label{Fs-gamma2}
F_\tau( \overline f ) =F_{t_1}(f_1)- \int_{t_2-\delta}^{t_1}  \big[  R(f_1(s)) -L(f_1(s),\dot{f}_1(s))\big]ds+ \int_{t_2-\delta}^{\tau}  \big[  R( \overline f (s)) -L( \overline f (s), \dot{\overline f}(s))\big]ds.
\end{equation} 
 Note that for all $\tau\in [t_2-\delta,t_2]$, $ \overline f (\tau)\in [x_2\wedge f_1(t_2-\delta), x_2\vee f_1(t_2-\delta)]$. Hence,  $ \overline f (\tau)$ is uniformly bounded in  $ [t_2-\delta,t_2]$, for fixed $(t_1,x_1)$ and $r$. We next note that
\[
|\dot{\overline f} (\tau)|=\f{|x_2-f_1(t_2-\delta)|}{ \delta}\leq\f{|x_1-x_2|}{\delta}+\frac{2|f_1(t_1)-f_1(t_2-\delta)|}{ t_1-t_2+\delta },\qquad \text{for all $\tau\in [t_2-\delta,t_2]$ }.
\]

Since $|\dot{f}_1|$ is uniformly bounded in $[0,t_1]$ thanks to Lemma \ref{lem:optimal-traj} and since $|x_1-x_2|\leq \delta$, we deduce that there exists a constant $C$ such that
\[
|\dot{\overline f} (\tau)|\leq C.
\]
 Therefore, the integrand terms in the r.h.s. of \eqref{Fs-gamma2} are bounded.

 Since $R$ and $L$ are locally bounded,  up to choosing $r$ to be a smaller constant, we obtain that, for all $\tau\in [t_2-\delta,t_2]$ and, 
 \[
 F_{{\tau}}( \overline f )\geq u_a(t_1,x_1)-C'\delta \geq u_a(t_1,x_1)-C'r > a.
 \]
 We   deduce that $(t_1,x_1)\in \Omega_a$ and hence $\Omega_a$ is an open set. Furthermore, we have,   for $r$ small enough
 \[
u_a(t_2,x_2)\geq F_{t_2}(\overline f)\geq u_a(t_1,x_1)- C'r\geq u_a(t_1,x_1)-\eta.
 \]

    (ii) {\bf $u_a$ is continuous on $\p \Omega_a$.}   We recall from \fer{upos} that $u\geq a$ in $\widetilde \Omega_a$. From the definition of $\Omega_a$, it is then immediate that $u_a=a$ on $\p \Omega_a\subset\widetilde\Omega_a\setminus{\Omega_a}$.    We  prove that $u_a$ is continuous on $\p \Omega_a$. Let $(\bar t,\bar x)\in \p \Omega_a$ and $(t_n,x_n)\in \Omega_a$ such that, as $n\to +\infty$, $(t_n,x_n)\to (\bar t,\bar x)$. Then, there exist optimal trajectories $f_n:[0,t_n]\to \R$, such that $f_n(t_n)=x_n$, $F_{s}(f_n)\geq a$, for all $s\in [0,t_n]$, and $u_a(t_n,x_n)=F_{t_n}(f_n)$. Similarly to the proof of Lemma \ref{lem:optimal-traj}, we deduce that $f_n$ converges along subsequences, as $n\to +\infty$, to an absolutely continuous trajectory $\overline f$ such that $\overline f(\bar t)=\bar x$, and $F_{s}(f_n)\to F_s(\overline f)$, for all $s\in [0,\bar t]$. Consequently, $F_s(\overline f)\geq a$ for all $s\in [0,\bar t]$ and $u(\bar t, \bar x) \geq F_{\bar t}(\bar f)=\lim_{n\to+\infty} u_a(t_n,x_n)\geq a $. Since $(\bar t, \bar x)\notin \Omega_a$, we deduce that $u_a(\bar t, \bar x)=a$, and hence $\lim_{n\to+\infty} u_a(t_n,x_n)=u_a(\bar t, \bar x)=a $.

 (iii) {\bf $u_a$ is locally Lipschitz continuous in $\overline\Omega_a$. } We first prove that $u_a$ is locally Lipschitz continuous in $\Omega_a$, for all $a>0$, in the following sense.
  Let $B_r\subset \Omega_a$ be a  ball of radius $r$.  We will prove that for $r$ chosen small enough, there exists a constant $C$ such that for all $(t_1,x_1)\in \overline B_r$ and $(t_2,x_2)\in \overline B_r$,
  \[
  |u_a(t_1,x_1)-u_a(t_2,x_2)|\leq C(|t_1-t_2|+|x_1-x_2|).
  \]
The proof follows similar arguments as in the proof of (i). We first choose $\delta =|t_1-t_2|+|x_1-x_2|$ and notice that $t_2-\delta <t_1$. From step (i),  $u_a$ is lower semi-continuous  and hence  $u_m:=\min_{(t,x)\in   \overline B_r} u_a(t,x)>a$. Let $f_1$ be the optimal trajectory such that $f_1(t_1)=x_1$ and $F_s(f_1)\geq a$ for all $s\in [0,t_1]$ and $F_t(f_1)=u_a(t_1,x_1)>a$. We define $\overline f$ as in \fer{eq:gammaO} and notice similarly to above that $\overline f$ and $|\dot{\overline f}|$ are bounded. We also notice that, for $\tau \in [0,t_2-\delta]$, $F_\tau(\overline f)\geq a$. We next use \fer{Fs-gamma2} and the boundedness of $\overline f$, $\dot{\overline f}$, $f_1$ and $\dot{f}_1$ to obtain that, for $\tau \in [t_2-\delta,t_2]$,
  \[
  F_{\tau}(\overline f)\geq F_{t_1}(f_1)-C\delta=u_a(t_1,x_1)-C\delta
  \geq u_m-C\delta.
  \]
  We deduce on the one hand that, for $r$ small enough, $  F_{\tau}(\overline f)\geq a$, for all $\tau \in [t_2-\delta,t_2]$. On the other hand, we have
  \[
  u_a(t_2,x_2)\geq F_{t_2}(\overline f)\geq u_a(t_1,x_1)-C\delta.
  \]
  The opposite inequality can be proved following similar arguments. We conclude that $u_a$ is Lipschitz continuous in ${\overline B_r}$.\\
  We next notice that the Lipschitz bound above only depends on the local bounds on $L$ and $R$. From the continuity of $u_a$ up to the boundary of $\Omega_a$, we deduce {that $u_a$ is indeed  locally Lipschitz  continuous in $\overline \Omega_a$.} 
  Since $u_a(t,x)=a$ in $\widetilde \Omega_a\setminus \Omega_a$, we deduce that $u_a$ is also locally Lipschitz in $\widetilde \Omega_a$. Finally, since the Lipschitz bound above only depends on the local bounds on $L$ and $R$, we deduce that there exist Lipschitz bounds which are locally uniform with respect to $a$.

  (v) {\bf The bound  from above.} From the definition of $u_a$ in \eqref{def:u-a} and the fact that $R$, $-L$ and $\beta_0$ are bounded from above, thanks to assumptions \fer{def:R} and \fer{as:beta0}, we obtain a uniform bound from above on $u_a$, locally in $t$ and globally in $x$.
 \end{proof}

 \begin{lemma}
 \label{lem:dynprog-vois}
 Let $(t,x)\in \Omega_a $  and $0<\tau<t$, with $\tau$ small enough. Then, we have
 \beq
 \label{u:dynPvois}
 u_a(t,x) =\sup_{\underset{f(t)=x}{f\in \mathrm{AC}[t-\tau,t]}} \int_{t-\tau}^t \big[  R(f(s)) -L(f(s),\dot{f}(s))\big]ds +u_a(t-\tau,f(t-\tau)).
 \eeq
 \end{lemma}
 \begin{proof}
 Let $f_0$ be an optimal trajectory such that $f_0(t)=x$ and $F_t(f_0)=x$, and $F_s(f_0)\geq a$, for all $s\in [0,t]$. We have 
 \[
 u_a(t,x)=\int_{t-\tau}^t \big[  R(f_0(s)) -L(f_0(s), \dot{f}_0(s))\big]ds+\int_{0}^{t-\tau} \big[  R(f_0(s)) -L(f_0(s),\dot{f}_0(s))\big]ds+\beta_0(f_0(0)).
 \]
 Since $F_s(f_0)\geq a$, for all $s\in [0,t]$ and in particular for all $s\in [0,t-\tau]$, we deduce that
 \[
 \begin{array}{rl}
 u_a(t,x) &\leq \int_{t-\tau}^t \big[  R(f_0(s)) -L(f_0(s),\dot{f}_0(s))\big]ds + u_a(t-\tau,f_0(t-\tau))\\
 &\leq \sup_{\underset{f(t)=x}{f\in \mathrm{AC}[t-\tau,t]}} \int_{t-\tau}^t \big[  R(f(s)) -L(f(s),\dot{f}(s))\big]ds +u_a(t-\tau,f(t-\tau)).
 \end{array}
 \]
 Let us now assume that $f_1:[t-\tau,t]\to \R$ is such that $f_1(t)=x$ and
  \beq
  \label{u-less-gamma}
 u_a(t,x) < \int_{t-\tau}^t \big[  R(f_1(s)) -L(f_1(s),\dot{f}_1(s))\big]ds + u_a(t-\tau,f_1(t-\tau)).
 \eeq
 Let also $f_2$ be an optimal trajectory such that $f_2(t-\tau)=f_1(t-\tau)$, $u_a(t-\tau,f_1(t-\tau))=F_{t-\tau}(f_2)$ and $F_s(f_2)\geq a$ for all $s\in [0,t-\tau]$. We then define
 \[
 f_3(s)=
 \begin{cases}
 f_2(s)& s\in [0,t-\tau],\\
 f_1(s)& s\in [t-\tau,t].
 \end{cases}
 \]
 Then it is immediate that $u_a(t,x)<F_t(f_3)$.  We will prove that $F_s(f_3)\geq a$, for all $s\in [0,t]$, which leads to a contradiction with the latter inequality and the definition of $u_a$. Notice that since $F_s(f_2)\geq a$, for all $s\in [0,t-\tau]$, we deduce that $F_s(f_3)\geq a$, for all $s\in [0,t-\tau]$. To prove this property for $s\in [t-\tau,t]$, we write
 \[
 F_s(f_3)>u_a(t,x)-\int_s^t  \big[  R(f_1(s)) -L(f_1(s),\dot{f}_1(s))\big]ds.
 \]
 Since $R$ is bounded from above and $L$ is bounded from below, we deduce that, for all $s\in [t-\tau,t]$ and $C$ a positive constant,
  \[
 F_s({f_3})> u_a(t,x)-C(t-s)\geq  u_a(t,x)-C(t-\tau).
 \]
 Since $u_a(t,x)>a$, choosing $\tau$ small enough, we obtain that $F_s(f_3)\geq a$ for all $s\in [t-\tau,t]$ and hence $F_s(f_3)\geq a$ for all $s\in [0,t]$.
 \end{proof}

 \medskip
 
 {\bf Proof of Theorem \ref{prop:var-to-visc}. } Thanks to Lemma \ref{lem:ucont}, $u_a$ is locally Lipschitz continuous in $\overline \Omega_a$ and due to Lemma \ref{lem:dynprog-vois} it satisfies \fer{u:dynPvois}. It is then immediate (see \cite[Section 3.3]{GB:13}) that $u$ is a viscosity solution to the Hamilton-Jacobi equation  \fer{eq:HJ} in $\Omega_a$. We conclude the proof using the uniqueness of locally Lipschitz and bounded from above viscosity solutions to \fer{eq:HJ}  \cite{GD.HF:00}.  Note that the uniqueness result in \cite{GD.HF:00} is given for a Hamilton-Jacobi equation, with a convex Hamiltonian, in the whole domain. However, the proof can be adapted to Hamilton-Jacobi equations, with a convex Hamiltonian, and with Dirichlet boundary conditions. Note also that when $\Omega_a$ is bounded, the uniqueness follows from more standard arguments as in \cite[Section 5]{GB:13}. 
 \qed

 We next prove  Lemma \ref{lem:ua-cont}.

 {\bf Proof of Lemma \ref{lem:ua-cont}.} (i) We first prove \eqref{eq:cont-Omega}. Let $\mathcal K\subset \R^+\times \R$ be a compact set. We define
 \[
 f^{\mathcal{K}}(a)=\int_{\mathcal K}\mathds{1}_{\widetilde\Omega_{a}}(t,x)dtdx.
 \]
 Since $(\widetilde \Omega_a)_a$ is a decreasing family of sets, we deduce that $ f^{\mathcal{K}}_a$ is a decreasing function with respect to $a$. Since a decreasing function has at most a countable set of discontinuity points, we deduce that $ f^{\mathcal{K}}(a)$ is continuous at almost every point $a$. At a continuity point $a_0$ of $f^{\mathcal{K}}$ we have
 \[
 \lim_{a\to a_0}\int_{\mathcal K}\mathds{1}_{\widetilde\Omega_{a}}(t,x)dtdx=\int_{\mathcal K}\mathds{1}_{\widetilde\Omega_{a_0}}(t,x)dtdx.
 \]
 Moreover, since $\Gamma_{a_0}=\bigcup_{a> a_0} \displaystyle\widetilde \Omega_a  $, 
 we deduce that 
 \[
  \int_{\mathcal K}\mathds{1}_{\Gamma_{a_0}}(t,x)dtdx.=\lim_{a\downarrow a_0}\int_{\mathcal K}\mathds{1}_{\widetilde\Omega_{a}}(t,x)dtdx=\int_{\mathcal K}\mathds{1}_{\widetilde\Omega_{a_0}}(t,x)dtdx.
 \]
 It follows that, for a.e. $a_0\in \R$,
 \[
    \int_{\mathcal  K} \mathds{1}_{\widetilde \Omega_{a_0}\setminus\Gamma_{a_0}}(t,x)dtdx=0.
 \]
Since this equality holds a.e. for all compact set 
$\mathcal K\subset \R^+\times\R$, we deduce that,  for a.e. $a_0\in \R$,
 \[
    \int_{\R^+\times\R} \mathds{1}_{\widetilde \Omega_{a_0}\setminus\Gamma_{a_0}}(t,x)dtdx=0.
 \]
 (ii) We next prove \eqref{ua-cont}. Let $\mathcal K\subset \R^+\times \R$, be a compact set. We then define
 \[
 g(a)=\int_{\mathcal K} \mathds{1}_{\widetilde \Omega_a} u_a(t,x)dtdx.
 \]
 Notice that $\mathds{1}_{\widetilde \Omega_a} u_a(t,x)$ is decreasing with respect to $a$. It hence converges, as $a\downarrow a_0$, to 
 $\mathds{1}_{\Gamma_{a_0}} v(t,x)$, for a certain function $v(t,x)$. Moreover, since $u_a$ is locally   Lipschitz in $\widetilde\Omega_a$ with respect to $t$ and $x$, with a locally uniform dependence on $a$, we deduce that $v(t,x)$ is indeed a continuous function in $\Gamma_{a_0}$. We also have
 \[
 \lim_{a\downarrow a_0} g(a)=\int_{\mathcal K}\mathds{1}_{\Gamma_{a_0}}v(t,x)dtdx.
 \]
We also notice that $g$ is a decreasing function with respect to $a$. Consequently, $g$ is continuous with respect to $a$, for almost every $a$. We deduce that, for almost every $a_0\in \R$, 
 \[
 \int_{\mathcal K} \mathds{1}_{\widetilde \Omega_{a_0}} u_{a_0}(t,x)dtdx=g(a_0)=\lim_{a\downarrow a_0} g(a)=\int_{\mathcal K}\mathds{1}_{\Gamma_a} v(t,x)dtdx.
 \]
From the monotonicity of $\Omega_a$ and $u_a$ we also obtain
 \[
 \mathds{1}_{\widetilde \Omega_{a_0}} u_{a_0}(\cdot,\cdot)\geq\mathds{1}_{\Gamma_{a_0}} v(\cdot,\cdot).
 \]
 Combining the lines above we obtain that, for almost every {$a_0\in \R$ and }$(t,x)\in \mathcal K$, 
 \[
  \mathds{1}_{\widetilde \Omega_{a_0}} u_{a_0}(t,x))=\mathds{1}_{\Gamma_{a_0}} v(t,x),
 \]
 and consequently, for almost every {$a_0\in \R$ and }$(t,x)\in \mathcal K \cap \Gamma_{a_0}$, 
 \[
 u_{a_0}(t,x)=v(t,x).
 \]
 Finally, since $u_{a_0}$ and $v$ are both continuous in $\Gamma_{a_0}$, we deduce that {for almost every $a_0\in \R$ and } for all
 $(t,x)\in \mathcal K \cap \Gamma_{a_0}$, 
 \[
 u_{a_0}(t,x)=v(t,x)=\lim_{a\downarrow a_0} u_a(t,x).
 \]
 Since this equality holds in $ \mathcal K \cap \Gamma_{a_0}$ for a.e. $a_0$ and any compact set $\mathcal K$, we deduce that it also holds for a.e. $a_0$ and for all  $(t,x)\in \Gamma_{a_0}$. Combining this property with \eqref{cont-ua-Omegac} we obtain \eqref{ua-cont}.
\qed

\bigskip

{{\bf Proof of Theorem \ref{cor:HJ}.} (i) We first prove the lower bound. From Theorem \ref{thm:borneinf} we deduce that 
\[
 \liminf_{K\rightarrow +\infty}   \frac{1}{\log K}\log N^{K,G^{x,\delta}_t}_{t} \geq \sup\{F_t(f); f\in G^{x,\delta}_t, \; \forall s\in [0,t],\; {F_s(f)}> 0\}. 
\]
Let $a>0$. Then, we have for all $a>0$ and $\delta>0$,
\[
\{f\in AC[0,t], \; f(t)=x,\; \forall s\in [0,t],\; {F_s(f)}\geq a\}\subset \{f\in G^{x,\delta}_t, \; \forall s\in [0,t],\; {F_s(f)}> 0\}. 
\]
Using \eqref{upos} we obtain that, for all $a>0$ and $\delta>0$,
\[
u_a(t,x)\leq \sup\{F_t(f); f\in G^{x,\delta}_t, \; \forall s\in [0,t],\; {F_s(f)}> 0\}. 
\]
Combining the properties above we conclude that
\[
\lim_{a\downarrow 0}u_a(t,x)\leq \liminf_{\delta\to 0}\liminf_{K\rightarrow +\infty}  \frac{1}{\log K} \log N^{K,G^{x,\delta}_t}_{t}.  
\]
(ii) We next prove the upper bound. From Theorem \ref{thm:bornesup} we obtain that 
\[
 \limsup_{K\rightarrow +\infty}   \frac{1}{\log K}\log N^{K,A^{x,\delta}_t}_{t} \leq \sup\{F_t(f); f\in A^{x,\delta}_t, \; \forall s\in [0,t],\; F_s(f)\geq 0\}. 
\]
Similarly to Lemma \ref{lem:optimal-traj}, and since the set $A_t^{x,\delta}$ is a closed set, there exists an optimal trajectory $f^\delta$ which maximizes $F_t(\cdot)$ in the set above, with $F_s(f^\delta)\geq 0$ for all $s\in [0,t]$. Moreover, $\|f^\delta\|_{W^{1,\infty}([0,t])}$ is  bounded uniformly with respect to $\delta$. Note that the $W^{1,\infty}$ bound is proved in  \cite{Fathi-book} and \cite[Lemma 6]{VC.KL:19} in the maximization problem with a fixed ending point. Here, we consider the trajectories with ending points in $[x-\delta,x+\delta]$. Therefore, if $f_0$ is the optimal trajectory, $f_0$ is also an optimal trajectory with ending point at $f_0(t)$ and the same result applies.   We deduce that
\[
\limsup_{\delta\to 0}\limsup_{K\rightarrow +\infty}   \frac{1}{\log K}\log N^{K,A^{x,\delta}_t}_{t} \leq 
\limsup_{\delta\to 0} F_t(f^\delta).
\]
Let $(f^{\delta_n})_n$ be a sequence of trajectories such that \[
\limsup_{\delta\to 0} F_t(f^\delta)=\lim_{n\to \infty} F_t(f^{\delta_n}).
\]
From the uniform bound in $W^{1,\infty}$, we deduce that, up to considering a subsequence, $f^{\delta_n}$ converges to a trajectory $f^0\in AC[0,t]$ such that
\[
\limsup_{\delta\to 0} F_t(f^\delta)=F_t(f^0),\qquad f^0(t)=x,\qquad F_s(f^0)\geq 0, \quad \forall s\in [0,t].
\]
It follows that
\[
\limsup_{\delta\to 0}\limsup_{K\rightarrow +\infty}   \frac{1}{\log K}\log N^{K,A^{x,\delta}_t}_{t} \leq u_0(t,x).
\]
\hfill$\Box$
}

\appendix

\section{Proof of the many-to-one formulas}
\label{sec:many-to-one}

\subsection{Proof of Proposition \ref{prop:FeynmanKac-bis} (i)}\label{app:manyto1}

		Let us give a simple  proof based on It\^o's formula. Let us first note that the intensity measure of ${Z}^K_t$, $\,
		\nu^K_{t}(dy)=\mathbb{E}_{\delta_{x}}\left[{Z}^K_{t}(dy)\right]\,$ defined for any $\varphi$ in $C_{b}(\mathbb{R})$ by
		\[
		\langle\nu^K_{t}, \varphi \rangle =\mathbb{E}_{\delta_{x}}\left[\langle {Z}^K_{t}, \varphi\rangle \right]
		\]
		is the unique weak solution of
		\begin{equation}
			\label{eq:prob2}
			\begin{cases}
				\partial_{t}\nu_{t}=\nu_t {\cal L}^K+ R\,\nu_{t},\\
				\nu_{0}=\delta_{x},
			\end{cases}
		\end{equation}
        where we denote by $\nu {\cal L}^K$ the adjoint of the operator ${\cal L}^K$ applied to the probability measure $\nu$.
		 Uniqueness of such a solution is proven as in Theorem 2.2 in \cite{HenryMeleardTran}.

	 Let us show that the r.h.s.\ term of \eqref{eq:MTO} also  satisfies \eqref{eq:prob2}. Uniqueness will yield the result.
		Let $\varphi$ in $C^1_{b}(\mathbb{R})$. 
		Applying It\^o's formula with jumps (e.g. \cite[Th.5.1]{ikedawatanabe}) to the semimartingale
		$\ \exp\left(\int_{0}^{t} R(X^K_s)ds \right)\varphi(X^K_{t})$,
		we have 
		\begin{align*}
		\exp\left(\int_{0}^{t} R(X^K_{s})ds \right)\varphi(X^K_{t}) 
        & -\varphi(X^K_{0}) 
		 -\int_{0}^{t}\exp\left(\int_{0}^{s} R(X^K_r)dr\right) {\cal L}^K\varphi(X^K_{s})\ ds\\
		& -\int_{0}^{t}\varphi(X^K_{s})R(X^K_s)\exp\left(\int_{0}^{s} R(X^K_r)dr \right)\ ds.		    
		\end{align*}
		is a square integrable martingale since $R$ is bounded. Taking the expectation, we obtain that
		\begin{multline}
		    \label{ito}
			\mathbb{E}_{x}\left[\exp\left(\int_{0}^{t} R(X^K_s)ds \right)\varphi(X^K_{t}) \right] =\varphi(x)\\
			+  \mathbb{E}_{x}\bigg[\int_{0}^t \exp\left(\int_{0}^{s} R(X^K_r)dr \right)  \bigg\{R(X^K_s)\,\varphi(X^K_{s}) + {\cal L}^K \varphi(X^K_{s})\bigg\}ds\bigg].
		\end{multline} 
		
		If we define the measure $\mu_{t}$ for any test function $\varphi\in \Co^1_b(\R)$ by
		$$\langle \mu_{t}, \varphi\rangle = \mathbb{E}_{x}\left[\exp\left(\int_{0}^{t} {R}(X^K_s)ds \right)\varphi(X^K_{t}) \right] ,$$
		we obtain from \eqref{ito}  that
		$$\langle \mu_{t}, \varphi\rangle = \langle \delta_{x}, \varphi\rangle +\int_{0}^t \langle \mu_{s},R\varphi+{\cal L}^K\varphi \rangle ds.$$
		This proves that the flow $(\mu_{t}, t\ge 0)$ is a weak solution of \eqref{eq:prob2} and the conclusion follows by uniqueness of solution of the equation.

\subsection{Proof of Proposition \ref{prop:forks}}\label{app:prop-forks}

Recall that $V^K_{[0,t]}$ is the set of individuals born before time $t$ and that $V^K_t$ is the set of individuals still alive at time $t$. 

Note that in our model, the trait $x$ of an individual $v$ remains constant during their life. The total rate of event for such an individual will be denoted here by
\[\Lambda(x)=b(x)+p(x)+d(x),\]
and the time $S_v$ where they dies or gives birth has a density $\Lambda(x)\exp\big(\Lambda(x)(s-S^v_0)\big)\ind_{s>S^v_0}$ with respect to Lebesgue's measure, conditionally on its birth time $S_v^0$. Let us note that
$$v\in V^K_s\Longleftrightarrow S_v^0\le s< S_v.$$

First, we have for all $v\in \mathcal{U}$,
\begin{align}
\lefteqn{   \E_{\delta_x}\Big[\ind_{v\in V^{K}_{[0,t]}} \Phi\big((X_{r\wedge S_v}^{K,v}, r\leq t),S_v \wedge t\big)\Big] }\nonumber\\
    = & \E_{\delta_x}\Big[\ind_{v\in V^{K}_{[0,t]}} \int_{S^0_v}^{+\infty} \Phi\big((X_{r\wedge s}^{K,v}, r\leq  t),s \wedge t \big) \Lambda(X^{K,v}_s)e^{-\int_{S^0_v}^s \Lambda(X^{K,v}_r)dr}\ ds\Big]\nonumber\\
    = & \E_{\delta_x}\Big[\ind_{v\in V^{K}_{[0,t]}} \int_{S^0_v}^{+\infty} \Phi\big((X_{r\wedge s}^{K,v}, r\leq  t),s\wedge t\big) \Lambda(X^{K,v}_s)\Big(\int_{s}^{+\infty} \Lambda(X^{K,v}_\tau) e^{-\int_{S^0_{{v}}}^\tau \Lambda(X^{K,v}_r)dr} d\tau\Big)\ ds\Big]\nonumber\\
    = & \E_{\delta_x}\Big[\ind_{u\in V^{K}_{[0,t]}} \int_{S^0_v}^{+\infty} \Big(\int_{S^0_v}^{\tau}\Phi\big((X_{r\wedge s}^{K,v}, r\leq  t),s\wedge t\big) \Lambda(X^{K,v}_s)  \ ds\Big) \Lambda(X^{K,v}_\tau) e^{-\int_{S^0_v}^\tau \Lambda(X^{K,v}_r)dr} \ d\tau\Big]\nonumber\\
    = & \E_{\delta_x}\Big[\ind_{v\in V^{K}_{[0,t]}} \int_{S^0_v}^{S_v} \Phi\big((X_{r\wedge s}^{K,v}, r\leq   t),s\wedge t\big) \Lambda(X^{K,v}_s)  \ ds\Big],\label{etape6}
\end{align}
where we used Fubini's theorem at the third equality, and where we recognized the distribution of $S_v$ to obtain the last equality. Then, summing \eqref{etape6} over $v\in \mathcal{U}$ entails 
    \begin{multline*}
        \E_{\delta_x}\left[\sum_{v\in {V}^{K}_{[0,t]}} \Phi(({X}^{K,v}_{r\wedge S_v},\ r\leq   t), S_v \wedge t) \right]\\
        \begin{aligned}
        = & \sum_{v\in \mathcal{U}} \E_{\delta_x}\Big[\ind_{v\in V^{K}_{[0,t]}} \int_{S^0_v}^{S_v} \Phi\big((X_{r\wedge s}^{K,v}, r\leq  t),s\wedge t\big) \Lambda(X^{K,v}_s)  \ ds\Big]\\
        = & \sum_{v\in \mathcal{U}} \E_{\delta_x}\Big[\int_{0}^{t} \ind_{v\in V^K_s} \ \Phi\big((X_{r\wedge s}^{K,v}, r\leq  t),s\big) \ \Lambda(X^{K,v}_s)  \ ds\Big]\\
        & \hspace{4cm} + \sum_{v\in \mathcal{U}} \E_{\delta_x}\Big[\ind_{v\in V^K_t} \ \Phi\big((X_{r}^{K,v}, r\leq  t), t\big) \int_{t}^{S_v}   \ \Lambda(X^{K,v}_s)  \ ds\Big]\\
        = & \int_0^t \E_{\delta_x}\Big[ \sum_{v\in V^K_s}\Phi\big((X_{r\wedge s}^{K,v}, r\leq  t),s\big) \ \Lambda(X^{K,v}_s)  \Big]\ ds\\
        & \hspace{4cm} + \E_{\delta_x}\Big[\sum_{v \in V^K_t} \Phi\big((X_{r}^{K,v}, r\leq  t), t\big) \Lambda(X^{K,v}_t) \, (S_v-t)\Big]\\
        = & \int_0^t \E_{\delta_x}\Big[ \sum_{v\in V^K_s}\Phi\big((X_{r\wedge s}^{K,v}, r\leq  t),s\big) \ \Lambda(X^{K,v}_s)  \Big]\ ds
         + \E_{\delta_x}\Big[\sum_{v\in V^K_t} \Phi\big((X_{r}^{K,v}, r\leq  t), t\big)\Big]\\
        \end{aligned}
    \end{multline*}
where we used in the third line that for $v\in V^K_t$ and $t\leq s<S_v$, $X^{K,v}_s=X_t^{K,v}$. For the last equality, we notice that 
\[\E\big[S_v-t \ |\ \mathcal{F}_t\big]=\frac{1}{\Lambda(X^{K,v}_t)}.\]
Using Proposition \ref{prop:FeynmanKac-bis}(ii), we obtain
\begin{multline*}
        \E_{\delta_x}\left[\sum_{v\in {V}^{ K}_{[0,t]}} \Phi(({X}^{K,v}_{r\wedge S_v},\ r\leq  t), S_v \wedge t) \right]\\
        \begin{aligned}= & \int_0^t  \E_x\Bigg[\Phi\big((X_{r\wedge s}^{K}, r\leq  t),s\big) \Lambda(X^K_s) \exp\Big(\int_0^s R(X^K_r)dr\Big)\Bigg]\ ds \\
        + & \E_x\Bigg[\Phi\big((X_r^{K}, r\leq t),t\big)  \exp\Big(\int_0^t R(X^K_r)dr\Big)\Bigg].
        \end{aligned}
    \end{multline*}
This ends the proof of Proposition \ref{prop:forks}.

\section{On Skorohod balls around absolutely continuous functions}\label{app:dsko-unif}

We prove here the inclusion stated in \eqref{etape21}. Let $f\in AC[0,t]$. Note that for any homeomorphism $\lambda$ of $[0,t]$, we have
\begin{align}\label{etape23}
\sup_{r\in [0,t]} \big|f(r)-g(r)\big| \leq & \sup_{r\in [0,t]} \big|f\circ \lambda(r)-g(r)\big| + \sup_{r\in [0,t]} |f\circ \lambda(r)-f(r)|.
\end{align}Thus, if $\dSko(f,g)<\varepsilon$, the first term in the right hand side can be made smaller than $2\varepsilon$ for a good choice of homeomorphism $\lambda$. By \eqref{def:dSko}, the latter homeomorphism can be chosen such that 
\[\sup_{r\in [0,t]} |\lambda(r)-r| \leq (e^\varepsilon-1)t.\]
Then, the second term in the right hand side of \eqref{etape23} is upper-bounded by the modulus of continuity $\omega(f,(e^\varepsilon-1)t)$, where 
\[\omega(f,\eta)=\sup_{|t-s|<\eta} \big|f(t)-f(s)\big|.\]
Since $f$ is absolutely continuous, and since $(e^{\varepsilon}-1)t$ converges to zero for $\epsilon\rightarrow 0$, the second term also converges to zero.
Gathering these equations gives \eqref{etape21}, for $\eta(\varepsilon)=2\varepsilon + \omega(f,(e^{\varepsilon}-1) t)$.

\section*{Acknowledgements}
We thank B. Mallein and P. Maillard for enlightening discussions and references.
This work is funded by the European Union (ERC, SINGER, 101054787 and ERC-2024-COG MUSEUM-101170884). Views and opinions expressed are however those of the author(s) only and do not necessarily reflect those of the European Union or the European Research Council. Neither the European Union nor the granting authority can be held responsible for them. This work has also been supported by the Chair ``Modélisation Mathématique et Biodiversité" of Veolia Environnement-Ecole Polytechnique-Museum National d’Histoire Naturelle-Fondation X. V.C.T. acknowledge the support of the R-CDP-24-004-C2EMPI project, funded by the French State under the France-2030 programme, the University of Lille, the Initiative of Excellence of the University of Lille, the European Metropolis of Lille.

\bigskip

{\footnotesize
\providecommand{\noopsort}[1]{}\providecommand{\noopsort}[1]{}\providecommand{\noopsort}[1]{}\providecommand{\noopsort}[1]{}

}

\end{document}